\documentclass[11pt,a4paper,reqno]{amsart}

\usepackage{todonotes}
\usepackage{amsmath}
\usepackage{amsfonts}
\usepackage{mathrsfs}  
\usepackage{amssymb}
\usepackage{amsthm}
\usepackage{hyperref}
\usepackage{graphicx}
\usepackage[noadjust]{cite}
\usepackage{caption}
\usepackage{subfig}
\usepackage{dutchcal}
\usepackage{bbm}
\usepackage{tikz}
\usetikzlibrary{decorations.pathreplacing}
\usepackage[T1]{fontenc}
\usepackage[utf8]{inputenc}
\usepackage{mathtools}
\usepackage{pdfpages}

\usepackage[top=2.92cm, bottom=2.92cm, left=2.5cm, right=2.5cm, headsep=0.2in]{geometry}

\usepackage{fancyhdr}

\DeclareMathOperator{\dist}{dist}

\theoremstyle{plain}
\newtheorem{theorem}{Theorem}[section]
\newtheorem{defi}[theorem]{Definition}
\newtheorem{corollary}[theorem]{Corollary}
\newtheorem{lemma}[theorem]{Lemma}

\newtheorem{remark}[theorem]{Remark}
\newtheorem{question}[theorem]{Question}
\newtheorem{claim}[theorem]{Claim}
\newtheorem{proposition}[theorem]{Proposition}

\newtheorem{conj}[theorem]{Conjecture}

\newtheorem{que}[theorem]{Question}

\newcommand{\RR} {\mathbb R}

\newcommand{\CC} {\mathbb C}

\newcommand{\EE} {\mathbb E}
\newcommand{\ZZ} {\mathbb Z}
\newcommand{\NN} {\mathbb N}

\newcommand{\NNN}{\mathcal{N}}

\newcommand{\pa} {\partial}
\newcommand{\Cal} {\mathcal}

\newcommand{\beq} {\begin{equation}}
\newcommand{\eeq} {\end{equation}}

\newcommand{\diam}{\operatorname{diam}}
\newcommand{\inrad}{\operatorname{inrad}}
\newcommand{\capacity}{\operatorname{cap}}

\renewcommand{\Im} {\operatorname{Im}}

\numberwithin{equation}{section}

\newcommand{\MM}[1]{{\color{teal} {\bf (MM: #1)}}}

\begin{document}
\title{
On the effects of small perturbation on low energy Laplace eigenfunctions}
\author{Mayukh Mukherjee}
\address{Indian Institute of Technology Bombay, Powai, Maharashtra 400076, India}
\email{mathmukherjee@gmail.com, mukherjee@mpim-bonn.mpg.de}
\author{Soumyajit Saha}
\address{Iowa State University, Ames, Iowa 50011, USA}
\email{ssaha1@iastate.edu}

\maketitle

\begin{abstract}
   We 
   investigate several aspects of the nodal geometry and topology of Laplace eigenfunctions, 
   with particular emphasis on the low frequency regime. This includes 
   investigations in and around the 
  Payne property, 
   opening angle estimates of nodal domains, saturation of (fundamental) spectral gaps etc., and behaviour of all of the above under small scale perturbations. We aim to highlight interesting aspects of spectral theory and nodal phenomena tied to ground state/low energy eigenfunctions, as opposed to asymptotic results. 
\end{abstract}

\tableofcontents

\section{Introduction and preliminaries}
Let $(M, g)$ be a compact Riemannian manifold. Consider the eigenequation
\begin{equation}\label{eqtn: Eigenfunction equation}
    -\Delta \varphi = \lambda \varphi, 
\end{equation}
where $\Delta$ is the Laplace-Beltrami operator given by (using the Einstein summation convention) 
$$
\Delta f = \frac{1}{\sqrt{|g|}}\pa_i \left( \sqrt{|g|}g^{ij}\pa_j f\right),
$$
where $|g|$ is the determinant of the metric tensor $g_{ij}$. In the Euclidean space, this reduces to the usual $\Delta = \pa_1^2 + \dots + \pa_n^2$. Observe that we are using the analyst's sign convention for the Laplacian, namely that $-\Delta$ is positive semidefinite.

If $M$ has a boundary, we will consider either the Dirichlet boundary condition
\begin{equation}\label{eqtn: Dirichlet condition}
    \varphi(x)=0, \hspace{5pt} x\in \pa M,
\end{equation}
or the Neumann boundary condition 
\begin{equation}\label{eqtn:Neumann_condition}
    \pa_\eta\varphi(x)=0, \hspace{5pt} x\in \pa M,
\end{equation}
where $\eta$ denotes the outward pointing unit normal on $\pa M$.
Recall that if $M$ has a reasonably regular boundary, $-\Delta_g$ has a discrete spectrum 
$$ 
0\leq\lambda_1 \leq \dots \leq \lambda_k \leq \dots \nearrow \infty,
$$
repeated with multiplicity with corresponding (real-valued $L^2$ normalized) eigenfunctions $\varphi_k$. Also, let  $\mathcal{N}_{\varphi_{\lambda}} = \{ x \in M: \varphi_{\lambda} (x) = 0\}$ denote the nodal set of the eigenfunction $\varphi_{\lambda}$. Sometimes for ease of notation, we will also denote the nodal set by $\NNN(\varphi_\lambda)$. Recall that any connected component of $M \setminus \mathcal{N}_{\varphi_{\lambda}}$ is known as a nodal domain of the eigenfunction $\varphi_{\lambda}$ denoted by $\Omega_{\lambda}$. These are domains where the eigenfunction is not sign-changing (this follows from the maximum principle). 

In this paper, we study the nodal topology and geometry associated to Laplace eigenfunctions, for both curved and flat domains, and with both Dirichlet and/or Neumann boundary conditions. Most of our results are related to the low frequency regime (low energy eigenfunctions), though some will also apply to the high frequency asymptotics, but no result is purely asymptotic in nature.

\subsection{Notational convention} When two quantities $X$ and $Y$ satisfy $X \leq c_1 Y$ ($X \geq c_2Y$) for constants $c_1, c_2$ dependent on the geometry $(M, g)$, we write $X \lesssim_{(M, g)} Y$ (respectively $X \gtrsim_{(M, g)} Y$). 
	Unless otherwise mentioned, 
	these constants will in particular be independent of eigenvalues $\lambda$. Throughout the text, the quantity $\frac{1}{\sqrt{\lambda}}$ is referred to as the wavelength and any quantity (e.g. distance) is said to be of sub-wavelength (super-wavelength) order if it is $\lesssim_{(M, g)} \frac{1}{\sqrt{\lambda}}$ (respectively $\gtrsim_{(M, g)} \frac{1}{\sqrt{\lambda}}$). 

First we recall a few basic facts and results that we would need in the sequel.
\subsection{Characterisation of eigenvalues}
We note that the Sobolev space $H^1(M)$ can be defined as the completion of $C^\infty(M)$ with respect to the inner product 
\begin{equation}
    \langle f, g \rangle_{H^1} := \langle f, g \rangle_{L^2(M)} +  \langle \nabla f, \nabla g \rangle_{L^2(M)} 
\end{equation}
where $f, g\in C^\infty(M)$. Next, we introduce a bilinear form in $H^1(M)$. Consider the bilinear form on $C^\infty(M)\times C^\infty(M)$,
\begin{equation*}
    D(f, g):= \langle \nabla f, \nabla g \rangle_{L^2(M)}.
\end{equation*}
Since $H^1(M)$ is the completion of $C^\infty(M)$ in the induced norm, given $f, g \in H^1(M)$, there exists sequence $\{f_i\}, \{g_i\} \in C^\infty(M)$ converging to $f, g$ in $H^1$ norm. Then we can define the bilinear form on $H^1(M)\times H^1(M)$ as
\begin{equation}
    D(f, g) =  \lim_{i\to \infty} \langle \nabla f_i, \nabla g_i \rangle_{L^2(M)}.
\end{equation}
Now we have the following characterisation of Laplacian eigenvalues.
\begin{theorem}\label{thm: Rayleigh chracterization}
For any $k\in \NN$, let $\{\varphi_1, \varphi_2, \cdots, \varphi_{k-1}\}$ be the first $k-1$ orthonormal eigenfunctions. Then for any $ f\in H^1(M),~ f\neq 0$ such that 
$$\langle f, \varphi_1 \rangle=\cdots = \langle f, \varphi_{k-1} \rangle=0$$  
we have
\begin{equation}\label{eqtn: Rayleigh chracterization}
    \lambda_k\leq   \frac{D(f, f)}{\| f \|^2}.
\end{equation}
Moreover, the equality holds if and only if $f$ is an eigenfunction corresponding to $\lambda_k$.
\end{theorem}
Note that for characterising the Dirichlet-Laplacian eigenvalues, the admissible function class in the above theorem is $H^1_0(M)$, which is the completion of $C^\infty_0(M)$ with respect to the induced norm from the above defined inner product. But, for Neumann boundary condition and manifolds without boundary, we use the above characterisation as it is.

Moreover, we also note that, in case of a boundaryless manifold or a manifold with Neumann boundary condition  we have that $\lambda_1=0$ and for a manifold with Dirichlet boundary, we have that $\lambda_1>0$. This follows from the above characterisation.

\subsection{Eigenfunctions and Fourier synthesis} In a certain sense, the study of eigenfunctions of the Laplace-Beltrami operator is the analogue of Fourier analysis in the setting of compact Riemannian manifolds. Recall that the Laplace eigenequation is the standing state for a variety of partial differential equations modelling physical phenomena like 
heat diffusion, wave propagation or Schr\"{o}dinger problems. 
Below, we note down this well-known method of ``Fourier synthesis'':
    \begin{eqnarray}
    \text{Heat equation} & (\pa_t - \Delta)u = 0 & u(t, x) = e^{-\lambda t}\varphi(x)\\
    \text{Wave equation } & (\pa^2_t - \Delta)u = 0 & u(t, x) = e^{i\sqrt{\lambda} t}\varphi(x)\\
    \text{Schr\"{o}dinger equation} & (i\pa_t - \Delta)u = 0 & u(t, x) = e^{i\lambda t}\varphi(x)
    \end{eqnarray}
Further, note that $u(t, x) = e^{\sqrt{\lambda }t}\varphi(x)$ solves the harmonic equation $(\pa^2_t + \Delta)u = 0$ on $\RR \times M$. In the interest of completeness, we include one last useful heuristic: if one considers the eigenequation (\ref{eqtn: Eigenfunction equation}) on metric balls of radius $\frac{\sqrt{\epsilon}}{\sqrt{\lambda}}$ and rescale to a ball of radius $1$, it produces an ``almost harmonic'' function (see Section $2$ of \cite{Ma} for more details). 

A motivational perspective of the study of Laplace eigenfunctions then comes from quantum mechanics (via the correspondence with Schr\"{o}dinger operators), where the $ L^2 $-normalized eigenfunctions induce a probability density $ \varphi^2(x) dx $, i.e., the probability density of a particle of energy $ \lambda $ to be at $ x \in M $.
Another physical (real-life) motivation of studying the eigenfunctions, dated back to the late 18th century, is based on the acoustics experiments done by E. Chladni which were in turn inspired from the observations of R. Hooke in the late 17th century. But what is surprising (at least to the present authors!) is that the earliest observation of these vibration patterns were made by G. Galileo in early 17th century. We quote below, from {\em Dialogues concerning two new sciences}, his observation about the vibration patterns on a brass plate.

{\em
``As I was scraping a brass plate with a sharp iron chisel in order to remove some spots from it and was running the chisel rather rapidly over it, I once or twice, during many strokes, heard the plate emit a rather strong and clear whistling sound; on looking at the plate more carefully, I noticed a long row of fine streaks parallel and equidistant from one another. Scraping with the chisel over and over again, I noticed that it was only when the plate emitted this hissing noise that any marks were left upon it; when the scraping was not accompanied by this sibilant note there was not the least trace of such marks. I noted also that the marks made when the tones were higher were closer together; but when the tones were deeper, they were farther apart.'' 
}

Of similar essence, the experiments of Chladni consist of drawing a bow over a piece of metal plate whose surface is lightly covered with sand. When resonating, the plate is divided into regions that vibrate in opposite directions causing the sand to accumulate on parts with no vibration. 

The study of the patterns formed by these sand particles (Chladni figures) were of great interest which led to the study of nodal sets and nodal domains. Below we discuss few properties of the Laplace eigenfunctions.

\subsection{Elementary facts about eigenvalues and eigenfunctions} We now collect some elementary facts about the eigenvalues and eigenfunctions of the Laplace-Beltrami operator. 
One well known global property of the eigenfunctions is the following theorem which gives an upper bound on the number of nodal domains corresponding to the $k$th eigenfunction $\varphi_k$.
\begin{theorem}[Courant's nodal domain theorem]
The number of nodal domains of $\varphi_k$ can be at most $k$. In other words, the total number of connected components of $M\setminus \NNN_{\varphi_k}$ is strictly less than $k+1$. 
\end{theorem}

\begin{remark}
$\varphi_1$ is always non sign changing. 
In the case of Laplace eigenfunctions, this can be easily observed by replacing $\varphi_1$ by $|\varphi_1|$, which is non-negative and using the variational characterisation (\ref{eqtn: Rayleigh chracterization}) above. 
\end{remark}
\begin{remark}
The multiplicity of $\lambda_1$ is always 1 i.e. $\lambda_1$ is simple (for a manifold without
boundary, $\lambda_1 = 0$ corresponding to the constant eigenfunctions). If not, then $\varphi_2$ has a constant sign, from the previous remark. This contradicts the fact that $\langle \varphi_1, \varphi_2 \rangle=0$, since $\varphi_1$ has a constant sign as well. As a result, $\lambda_1$ is characterized as being the only eigenvalue with eigenfunction of constant sign. For significantly more general operators (like Schr\"{o}dinger operators), the result is still true, but this requires the use of the Krein-Rutman theorem.
\end{remark}

\begin{remark}
The above two remarks in conjunction with the Courant nodal domain theorem imply that $\varphi_2$ has exactly two nodal domains. Moreover, any $\varphi_k$ has at least two nodal domains for $k\geq 2$. 
\end{remark}

\begin{theorem}[Domain Monotonicity]
Suppose $\Omega_1 \subseteq \Omega_2 \subseteq M$. Then their fundamental Dirichlet eigenvalues satisfy
\beq
\lambda_1(\Omega_2) \leq \lambda_1(\Omega_1), 
\eeq
and the above inequality is strict if the set $\Omega_2 \setminus \Omega_1$ has positive capacity.
\end{theorem}

\begin{theorem}[Wavelength density]
For any $(M, g)$, there exists a constant $C>0$ (depending on $g$) such that every ball of radius bigger that $C/\sqrt{\lambda}$ intersects with the nodal set  corresponding to $\varphi_\lambda$. 
\end{theorem}
Note that the above theorem tells us that the nodal domains cannot be too ``fat'' which leads to a natural question:  How ``fat'' or ``thin'' can a nodal domain be?\footnote{Admittedly, the question becomes more interesting when we investigate this question for higher (asymptotic) eigenfunctions.} One way of interpreting the ``thickness'' of a domain is to measure the size of the largest ball that can be inscribed inside that domain. The inner radius estimates of nodal domains have been studied for years with fascinating results. 
We state the following result which is the culmination of several works:
\begin{theorem}[\cite{H, Cr, O, Ta, Ma, Ma1, Lie, GM}]
Let $(M, g)$ be a compact Riemannian manifold of dimension $n$, and let $\Omega_\lambda$ be a nodal domain for $\varphi_\lambda$. If $n = 2$, then we have that 
\beq
\inrad \left(\Omega_\lambda\right) \sim_{(M, g)} \frac{1}{\sqrt{\lambda}}.
\eeq
If $n \geq 3$, then 
one has
\begin{equation}
    \frac{1}{\lambda^{\frac{1}{2n}+\frac{n-1}{4}}}\lesssim_{(M, g)} \inrad(\Omega_\lambda) \lesssim_{(M, g)} \frac{1}{\sqrt{\lambda}}.
\end{equation}
Furthermore, let $x_0$ be a maximum point of $\varphi_\lambda$ on $\Omega_\lambda$. Then the ball of inner radius can be centered at $x_0$. 
\end{theorem}
We end this section by giving an overview of the paper.
\subsection{Overview of the paper} Here we take the space to list some of our main results.

First, we discuss the stability of topological properties of the first Dirichlet nodal set  (or, nodal set for any of the second Dirichlet eigenfunctions) under small perturbations. Some of these results are natural extensions to our previous work in \cite{MS}. Among other results, we prove that satisfying the strong Payne property\footnote{See Definition \ref{def:Payne_prop} below.} (or not satisfying) are both open conditions in a one-parameter family of perturbations of a given bounded domain $\Omega \subseteq \RR^n$: this is Proposition \ref{prop: SP/NP open condition} below. 
Observe that the results hold in all dimensions, and not restricted to planar domains only.

After a general laying of foundations, we discuss behaviour of the second Dirichlet eigenfunction in the connector or ``handle'' region of thin dumbbell domains (for a proper definition, see Subsection \ref{subsec: dumbbells} below). 
  In particular, we check the  validity of the Payne property (see Definition \ref{def:Payne_prop} below). We quote the result:
\begin{theorem}\label{thm: narrow connectors}
Consider two bounded 
domains $\Omega_1$, $\Omega_2\subset \RR^n~(n\geq 2)$ with $C^\infty$ boundary and a one parameter family of smooth dumbbells $\Omega_\epsilon$ (as described in Subsection \ref{subsec: dumbbells}) 
whose connector widths go to zero as $\epsilon\to 0$. Assume that $\Omega_i, i = 1, 2$ have simple second eigenvalues. Let $\lambda_{2}^{\Omega_i}, \lambda_{2,\epsilon}$ denote the second eigenvalues of $\Omega_i$, $\Omega_\epsilon$ corresponding to eigenfunctions $\varphi_2^{\Omega_i}$, $\varphi_{2,\epsilon}$ respectively, $i = 1, 2$. Assume 
that the connector does not intersect $\NNN(\varphi_2^{\Omega_i})\cap \pa \Omega_i$  and $\Omega_i$ do not have the same first or second Dirichlet eigenvalues.  

If $\lambda_{2,\epsilon} \to \lambda_2^{\Omega_1}$  (without loss of generality) and $\Omega_1$ satisfies the strong Payne property, 
then for sufficiently small $\epsilon>0$, $\Omega_\epsilon$ satisfies the strong Payne property as well. 
\end{theorem}

We note that the simplicity assumption of the second eigenvalue is true up to generic perturbations (see \cite{U}, and also Theorem \ref{thm:generic_spec} below).



Next, we start a discussion about negative results surrounding the Payne property. In Subsection \ref{subsec:counter_Four}, we first describe the 
two-dimensional counterexample in \cite{HHN} and the higher dimensional counterexample in \cite{F}. Next we give a new counterexample to the Payne property in higher dimensions, which has the merit of being simply connected as well. Our example is a perturbation of the base domain provided by Fournais in \cite{F}.  
We believe that conceptually our example might be slightly simpler than the counterexample in \cite{Ke}. We then finally indicate how to jazz it up to get a domain with prescribed topological complexity which violates the Payne property. 
More precisely,

\begin{theorem}\label{thm: satisfying NP}
Let $\Omega\subset \RR^n$ ($n\geq 3$) be any bounded domain with smooth boundary. Then there exists a family of (non-convex) domains $\Omega_\epsilon$ with the same fundamental group as $\Omega$ such that for small enough $\epsilon$, $\Omega_\epsilon$ does not satisfy the Payne property.
\end{theorem}

Next, we begin an investigation into some geometric properties of the nodal set of low energy eigenfunctions. Our discussion is mainly centred around the angular properties of the nodal set where it meets the boundary. First, 
to set the stage, we prove the following for the case of the nodal set intersecting the boundary (the interior case has already been addressed in \cite{GM1}):
\begin{theorem}\label{thm:ICC}
    Suppose $\NNN_\varphi$ intersects the boundary of the manifold $\pa M$ at a point $p$. When $\text{dim }M = 3$, the nodal set $\NNN_\varphi$ at $p$ satisfies an interior cone condition (see Definition \ref{def:int_cone_cond} below) with angle $\gtrsim \frac{1}{\sqrt{\lambda}}$. When $\text{dim }M = 4$, $\NNN_\varphi$ at $p$ satisfies an interior cone condition with angle $\gtrsim \frac{1}{\lambda^{7/8}}$. Lastly, when $\text{dim }M \geq 5$, $\NNN_\varphi$ at $p$ satisfies an interior cone condition with angle $\gtrsim \frac{1}{\lambda}$.
\end{theorem}

Now, we give a sharpened version of Theorem \ref{thm:ICC} above about the angle of intersection at the nodal critical points provided the nodal set is locally the intersection of two hypersurfaces. 
\begin{theorem}\label{thm:angle_estimate_high_dim}
    Let 
    $p \in M$, where $M$ is a compact manifold of dimension $n \geq 3$. Let $p$ lie at the intersection of two nodal hypersurfaces $M_1, M_2$. 
    Let $\eta_1, \eta_2\in S^{n-1}$ be two unit normal vectors to $M_1$ and $M_2$ at 
    $p$. If the order of vanishing of $\varphi_\lambda$ at $p$ is $n_0$, then the angle between $M_1$ and $M_2$ at $p$,  $\arccos{\langle\eta_1, \eta_2 \rangle}\in P$, where
$$P=\left\{\frac{p}{q}\pi: q=1, 2, \cdots n_0, p= 0, 1, \cdots, q\right\}.$$
\end{theorem}
Roughly, since the angle of intersection cannot change continuously, this result should be seen as a ``perturbation resistant'' nodal geometry phenomenon, and puts a constraint on the possible Payne configurations before the bifurcation.

In the last section of our paper, we start focusing  more on low energy spectral theory (properties of eigenvalues) with the help of our perturbation theoretic tools. To set up the stage, we first give a proof of a well-known result of Uhlenbeck that generic smooth perturbations of a domain have simple Dirichlet spectrum (this is Theorem \ref{thm:generic_spec} below). The ideas involved in our proof are based on \cite{GS} and are well-known by now, though we have not seen this exact proof in literature. 
Finally,  we start investigating 
our main target to address the question of whether the fundamental gap $\lambda_2 - \lambda_1$ is attained on convex domains. Here is what we prove:
Let $\mathfrak{P}$ denote the class of strictly convex $C^2$-planar domains. Then, we have that 
\begin{theorem}\label{thm:fund_gap_perturb}
    Let $\Omega \in \mathfrak{P}$ 
    with  diameter $D = 1$ and inner radius $\rho$. There exists a universal constant $C \ll 1$ such that if $\rho \leq C$, $\Omega$ cannot be the minimiser of the fundamental gap functional in $\mathfrak{P}$. 
\end{theorem}
The proof uses a significant portion of the ideas on perturbation theory developed so far in Section \ref{sec:perturbation}. We are unable yet to get a result in the full class of all convex domains. However, the popular belief in the community seems that the fundamental gap is not saturated in the class of all convex domains, and any infimising sequence for $\lambda_2 - \lambda_1$ (under the normalisation $D = 1$) should degenerate to a line segment. This indicates that the ``correct regime'' to look for in the search for minimisers is the class of narrow convex domains. Section \ref{sec:perturbation} is interspersed with open questions/speculations/conjectures which we believe to be of further/related interest, at any rate to the present authors!

We add an Appendix at the end where we investigate the interconnection of multiplicity of eigenvalues and the topology of the first nodal set. We begin by checking that results in \cite{Do, Gi} implying topological complexity of  the domain from ``detachment'' or ``non-intersection'' of nodal sets of eigenfunctions still hold true for Euclidean domains with appropriate boundary conditions. Further, we show that for a broad class of non-convex simply connected domains, the multiplicity of the second Dirichlet eigenvalue is $\leq 2$. These are directly related to the Payne property via an insight from \cite{Lin}, and should be of some value in the future studies of the Payne property. 



\section{Stability, Payne property and some preliminary results}\label{sec: stability}
In this section we will first look into a certain aspect of the nodal sets that remains stable under perturbation. Note that many aspects of nodal sets of Laplace eigenfunctions are rather unstable under perturbation, which normally disallows perturbative techniques (like normalized Ricci flow and related geometric flows etc.) in the study of nodal geometry. However, one is inclined to ask the question that if the perturbations are ``small enough'', are there certain ``soft'' properties of the nodal set that are still reasonably stable? This was answered in \cite{MS} where we proved that if the perturbation is of subwavelength scale, then the nodal sets do not ``see it''. The one-parameter family of perturbations considered in \cite{MS} was inspired from the construction by Takahashi in \cite{Tak}, but our proof is not restricted to that. Below we provide a more general setting under which the stability arguments go through. 

\subsection{Stability under general perturbations}
Suppose $\Omega_t$ is a one-parameter family of perturbations to the domain $\Omega_{t_0}$ with a simple spectrum such that for all $k = 1, \cdots,$ we have 
\begin{equation}\label{eq:eigenval_conv}
    \lim_{t \to t_0} \lambda_k(\Omega_t) = \lambda_k(\Omega_{t_0}).
\end{equation}
Fixing $k$, additionally assume that the perturbations happen away from the $k$-th nodal set of $\Omega_{t_0}$ and consider  
the following $C^\infty-$convergence\footnote{Actually, it turns out that for almost all of our applications, only $C^0$-convergence would suffice.} of eigenfunctions $\varphi_{k,t}\in C^{\infty}(\Omega_t)$, where $\varphi_{k, t}$ is defined as the $k$-th Dirichlet eigenfunction of $\Omega_t$:
\begin{equation}
    \lim_{t \to t_0} \varphi_{k,t}= \varphi_{k,t_0} \text{ on } 
    \tilde{\Omega} := \{ x \; : \; \exists \varepsilon_x > 0 \text{ such that } x \in \bigcap_{t_0 - \varepsilon, t_0 + \varepsilon}\Omega_t \; \forall \; \varepsilon < \varepsilon_x\}. 
\end{equation}
$\tilde{\Omega}$ defined above denotes the unperturbed part of $\Omega_{t_0}$. 
All the eigenfunctions involved are assumed to be $L^2$-normalized, i.e.,
$$
\|\varphi_{k,t}\|_{L^2(\Omega_{t})} = \|\varphi_{k,t_0} \|_{L^2(\Omega_{t_0})} = 1.
$$
Then we have the following
\begin{lemma}\label{lem:lim_M_1}
Let $\NNN(\varphi_{k,t})$ and $\NNN(\varphi_{k,t_0})$ denote the nodal sets corresponding to 
$\varphi_{k,t}$ and $\varphi_{k,t_0}$ respectively. Consider a sequence of points $\{x_i\}$ such that for each $i$, $x_i\in \NNN(\varphi_{k,t_i})\cap 
\tilde{\Omega}$. If the limit $x$ of $\{x_i\}$ exists, then $x\in \NNN(\varphi_{k,t_0})$.
\end{lemma}
Observe that there is no requirement of convexity or boundary regularity.
\begin{proof}
Denote $\tilde{\varphi}_{k,t} = \varphi_{k,t}|_{\tilde{\Omega}}$. We apply the convergence  
 $$\varphi_{k,t} \to \varphi_{k,t_0} \text{ in  } C^0 (\tilde{\Omega})$$
 to obtain,
 $$|\varphi_{k,t_0}(x_i)|= |\varphi_{k,t_0}(x_i)- \tilde{\varphi}_{k,t_i}(x_i)|\leq \|\varphi_{k,t_0} - \tilde{\varphi}_{k,t_i}\|_{C^0 (\tilde{\Omega})} \to 0 \hspace{5pt} \text{ as } i\to \infty.$$
  Now, $\varphi_{k,t_0}$ being a continuous function, $x_i\to x$ implies $\varphi_{k,t_0}(x_i)\to \varphi_{k,t_0}(x)$. Therefore, $\varphi_{k,t_0}(x)=0$ i.e. $x\in \NNN(\varphi_{k,t_0})$. 
\end{proof}

Now we have the following
\begin{lemma}\label{lem:nod_set_eventually}
If $\NNN(\varphi_{k,t_0})\subset \tilde{\Omega}$, then for $t$ close enough to $t_0$, the nodal set $\varphi_{k,t}$ is fully contained inside $\tilde{\Omega}$.
\end{lemma}
\begin{proof}
The proof 
follows by routine modifications to the proof of Lemma 3.6 of \cite{MS}, hence we skip the details.
\end{proof}

Next, we also recall the following convergence theorem from \cite{HP} which will play a crucial role in the upcoming discussion.
\begin{theorem}[Theorem 2.2.25, \cite{HP}]\label{thm: Henrot_conv_theo}
Let $K_n$ be a sequence of compact sets contained in a fixed compact set $B$. Then there exist a compact set $K$ contained in $B$ and a subsequence $K_{n_k}$ that converges in the sense of Hausdorff to $K$ as $k\to \infty$.
\end{theorem}
Although the nodal sets are not stable under perturbation (in the Hausdorff sense) as pointed above, we will prove below that certain nodal configurations of the first nodal set remain stable under small enough perturbations. Before moving forward, let us look at a celebrated conjecture of Payne which is related to the topology type of the first nodal set for bounded planar domains. A substantial portion of our discussion in this paper will revolve around this conjecture. The conjecture states the following: 
\begin{conj}[Conjecture 5, \cite{P}]
For a bounded domain $\Omega\subset \RR^2$, the second eigenfunction of the Laplacian with Dirichlet boundary condition 
does not have a closed nodal line.   
\end{conj}
We will refer this conjecture as the Payne conjecture or the nodal line conjecture throughout our text. The second nodal domains represent a 2-partition of $\Omega$ minimising the spectral energy, that is 
$$\lambda_2(\Omega)=\inf\{\max\{\lambda_1(\Omega_1), \lambda_1(\Omega_2)\}:\Omega_1, \Omega_2\subset\Omega \text{ open, } \Omega_1\cap \Omega_2=\emptyset,\; \overline{\Omega_1\cup \Omega_2}=\overline{\Omega}\}$$
where the infimum is attained only when $\Omega_1, \Omega_2$ are the nodal domains of some $\varphi_2$. One idea behind the above conjecture is that it would be suboptimal from the perspective of energy minimisation to have one nodal domain concentrated somewhere in the interior of $\Omega$, with the other occupying its boundary. Liboff, in \cite{Li}, conjectured that the nodal surface of the first excited state of a three-dimensional convex domain intersects its boundary in a single simple closed curve. The conjecture is analogous to that of Payne in dimension 3.

\subsection{Some previous work on topology of first nodal sets} 
Now, let us look at some progress made on the above conjecture in a chronological order. 

\cite{P1} addressed the conjecture provided the domain $\Omega \subseteq \RR^2$ is symmetric with respect to one line and convex with respect to the direction vertical to this line. \cite{Lin} following a similar approach proved the conjecture provided the domain $\Omega\subseteq\RR^2$ is smooth, convex and invariant under a rotation with angle $2\pi p/q$, where $p$ and $q$ are positive integers. Both the proofs rely heavily on the symmetry of the domain. In \cite{LN}, Lin and Ni provided a counter-example of the nodal domain conjecture for the Dirichlet Schrodinger eigenvalue problem. For each $n\geq 2$, they construct a radially symmetric potential $V$  in  a ball so that the nodal domain conjecture is violated.  
In  1991, Jerison proved in \cite{J} that the conjecture is true for long thin convex sets in $\RR^2$. More specifically, there is an absolute constant $C$ such that given a convex domain $\Omega\subset\RR^2$ with $\frac{\diam(\Omega)}{\inrad(\Omega)}\geq C$, we have that the nodal set corresponding to the second eigenfunction 
intersects the boundary at exactly two points. Here $\inrad(\Omega)$ denotes the radius of the largest ball that can be inscribed in $\Omega$ and $\diam(\Omega)$ denotes the diameter.  In the following year, Melas relaxed the condition of ``long and thin'' in \cite{M} and proved the conjecture for any bounded convex domain $\Omega$ in $\RR^2$ with $C^{\infty}$ boundary. Alessandrini further relaxed the $C^\infty$-boundary condition and proved the conjecture for general convex planar domains in \cite{A}. 

To the extent of our knowledge, M. Hoffmann-Ostenhof, T. Hoffmann-Ostenhof, and Nadirashvili in \cite{HHN} provided the first counter-example of the Payne conjecture in $\RR^2$ for the case of Dirichlet Laplacian. 
We outline the basic idea of their construction in Subsection \ref{subsec:counter_Four}  below. We mention in passing that boundedness of the domain is crucial for results of the Payne type (see \cite{FK1}).

Regarding the topological properties of the first nodal set in higher dimensions, Jerison \cite{J2} extended his result for long and thin convex sets in higher dimensions 
(see also related follow up work in \cite{Da, FK2, KT}). 
In \cite{F}, Fournais extended the result of \cite{HHN} in higher dimensions and proved that the first nodal set does not intersect the boundary (we outline his construction in Subsection \ref{subsec:counter_Four} below). The domain constructed by Fournais was not topologically simple, which was later addressed in \cite{Ke}. 
Recently Kiwan, in \cite{Ki}, proved the nodal domain conjecture for domains which are of the form $A\setminus B$ where $A$ and $B$ have sufficient symmetry and convexity.


\subsection{Payne property and perturbation of domains}

Let $\varphi$ be a Dirichlet eigenfunction for a bounded domain $\Omega\subset\RR^n$ with smooth boundary. For any $p\in \pa \Omega$, $p\in \NNN_{\varphi}$ if and only if $\frac{\pa \varphi}{\pa \eta}=0$, where $\eta$ denotes the outward normal at $p$. The proof for dimension $n = 2$ is covered in Lemma 1.2 of \cite{Lin}, and one can check that a similar proof is true in higher dimensions as well. Let $x = (x_1,\dots, x_n) =: (x', x_n)$, and let the domain $\Omega$ be tangent to the $x'$-hyperplane at the origin. If $\frac{\pa \varphi}{\pa \eta}(0) \neq 0$, then by the implicit function theorem, $x_n$ is uniquely solvable as a function of $x'$ in a neighbourhood of $0$, which means that the only zeros of $\varphi$ near the origin occur on $\pa\Omega$.  The converse case is addressed by a variant of the Hopf boundary principle (see Lemma H of \cite{GNN}). 

Consider any Dirichlet eigenfunction $\varphi$ whose nodal set $\NNN_\varphi$ divides $\Omega$ into exactly two nodal domains. In particular, any first nodal set (nodal set corresponding to some second eigenfunction) always divides the domain $\Omega$ into exactly two components.  Then we have the following three cases.
\begin{enumerate}
    \item[(SP)] If $\displaystyle\frac{\pa \varphi}{\pa \eta}$ changes sign on the boundary, then $\overline{\NNN}_\varphi\cap \pa \Omega\neq \emptyset$ and $\overline{\NNN}_\varphi$ divides at least one component of $\pa \Omega$ into exactly two components.

    \item[(WP)] If $\displaystyle \frac{\pa \varphi}{\pa \eta}\geq 0$ (without loss of generality) on the boundary with at least one point $x\in\pa \Omega$ such that $\displaystyle \frac{\pa \varphi}{\pa \eta}(x)=0$, then $\overline{\NNN}_\varphi\cap \pa \Omega\neq \emptyset$ but $\pa \Omega \setminus \overline{\NNN}_\varphi$ has same number of connected components as $\pa \Omega$. 
    
    \item[(NP)] If $\displaystyle \frac{\pa \varphi}{\pa \eta}> 0$ (without loss of generality) on the boundary, then $\overline{\NNN}_\varphi\cap \pa \Omega=\emptyset$.
\end{enumerate}
\begin{defi}\label{def:Payne_prop}
 We say that any second eigenfunction $\varphi$ satisfies the Payne property if the nodal set of $\varphi$ intersects the boundary $\pa \Omega$, that is either (SP) or (WP) is true. We say that $\varphi$ satisfies the strong Payne property if only (SP) is true. Also, we say that $\Omega$ satisfies the (strong) Payne property if every second eigenfunction $\varphi$ of $\Omega$ satisfies the (strong) Payne property. 
\end{defi}

As a consequence, we have the following
\begin{proposition}\label{prop: SP/NP open condition}
Satisfying the property (SP) or (NP) is an open condition. 
\end{proposition}

\begin{proof}
By Lemma \ref{lem:nod_set_eventually}, we know that $\NNN(\varphi_{2,\epsilon})$ is eventually inside $\tilde{\Omega} \subset \Omega_{t_0}$. 
By precompactness in Hausdorff metric as explained in Theorem \ref{thm: Henrot_conv_theo}, 
one can extract a subsequence called  $\NNN(\varphi_{2,\epsilon_i})$, 
which converges to a set $X \subset \Omega_{t_0}$ 
in the Hausdorff metric. By Lemma \ref{lem:lim_M_1}, we already know that $X \subset \NNN(\varphi_{2,t_0})$. 
It follows that for $i$ large enough, $\NNN(\varphi_{2,\epsilon_i})$ is within any $\delta$-tubular neighbourhood of $\NNN(\varphi_{2,t_0})$. 

 \begin{figure}[ht]
\centering
\includegraphics[scale=0.16]{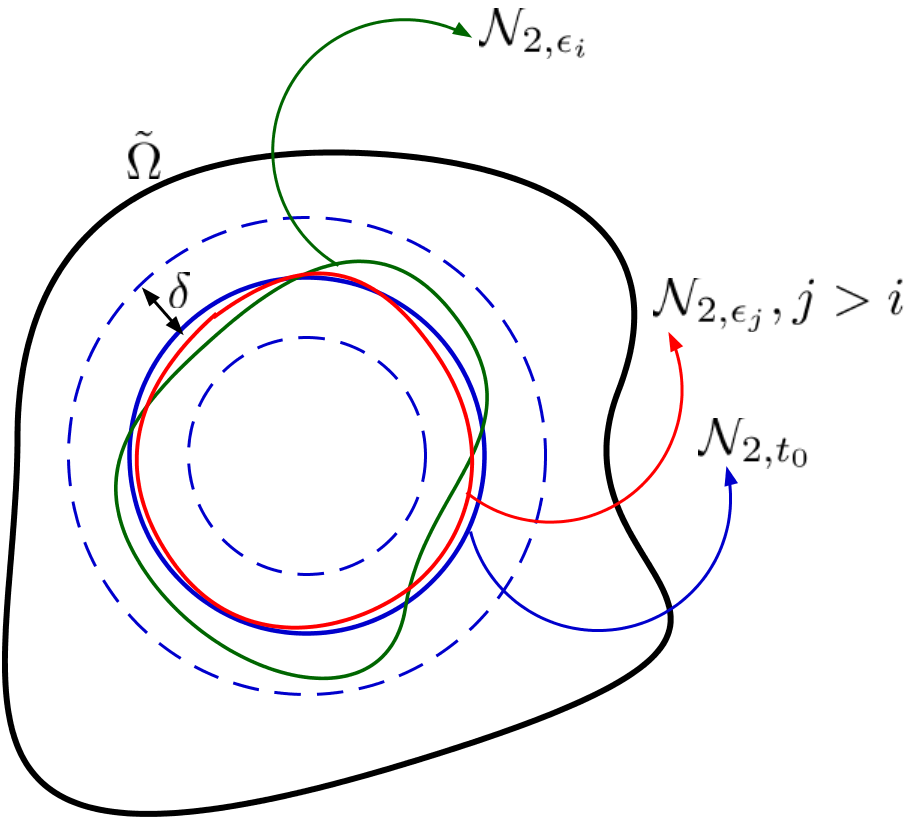}
\caption{Property (NP) is an  open condition}
\label{fig:NP condition is open}
\end{figure}
     

Let $\NNN(\varphi_{2,t_0})$ satisfy (NP). Then $\NNN(\varphi_{2,t_0})$ does not intersect the boundary $\pa \tilde{\Omega}$. For small enough $\delta$, the $\delta$-tubular neighbourhood of $\NNN(\varphi_{2,t_0})$ does not intersect $\pa \tilde{\Omega}$. This implies that given such a $\delta$, for large enough $i$, $\NNN(\varphi_{2,\epsilon_i})$ does not intersect $\pa \tilde{\Omega}$. More specifically, $\NNN(\varphi_{2,\epsilon_i})\cap \Omega_{\epsilon_i}=\emptyset$.

Now  assume that $\NNN(\varphi_{2,t_0})$ satisfies (SP). If possible, let (SP) is not an open condition, that is there exists a subsequence $\{k\}\subset \{i\}$ such that $\NNN(\varphi_{2,\epsilon_k})$ does not satisfy (SP). This means that one of nodal domains of the second Dirichlet eigenfunction of $\Omega_{\epsilon_k}$ is within any $\delta$-tubular neighbourhood of $\NNN(\varphi_{2,t_0})$ and the volume of such a tubular neighbourhood is going to $0$ as $\delta \searrow 0$. 
 \begin{figure}[ht]
\centering
\includegraphics[scale=0.17]{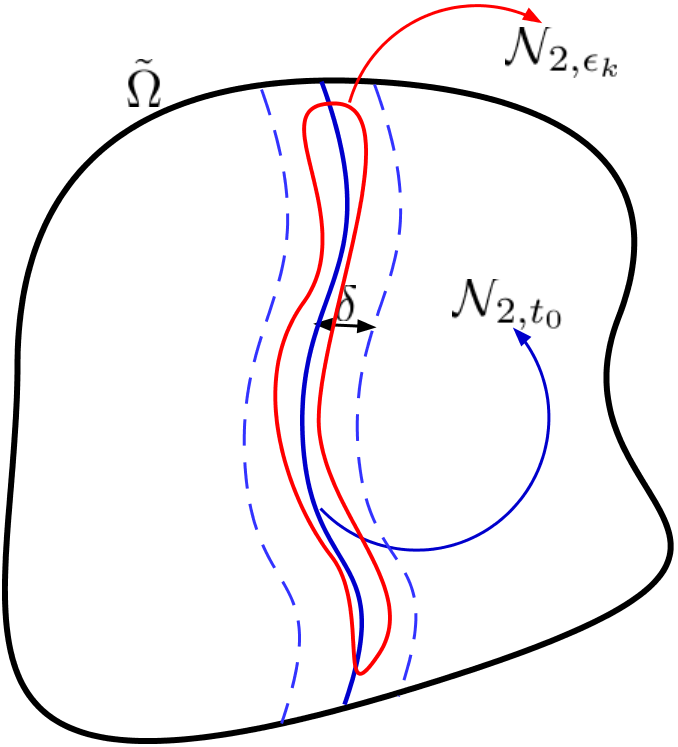}
\caption{If property (SP) is not an  open condition}
\label{fig:SP condition is open}
\end{figure}
This will contradict the Faber-Krahn inequality (or the inner radius estimate for the second nodal domain of $\Omega_{\epsilon_k}$), and imply that for large enough $i$, $\NNN(\varphi_{2,\epsilon_i})$ intersects the boundary. Moreover, if the first nodal set is a submanifold, then using Thom's isotopy theorem (see \cite{Ab}, Section 20.2) one can conclude that for large enough $i$, $\NNN(\varphi_{2,\epsilon_i})$ is diffeomorphic to $\NNN(\varphi_{2,t_0})$. 
\end{proof}

In essence, the goal of the proof described above is two fold: to prove the stability of the boundary Neumann data in case that the limiting domain satisfies (SP) or (NP), and to show that the first nodal set of the perturbed domains cannot perturb (in Hausdorff sense) too much from the limiting nodal set. 
More precisely, for sufficiently small perturbations the first nodal set of the perturbed domains should lie in a $\delta$-neighbourhood ($\delta$ however small) of the first nodal set of the limiting domain given that the limiting nodal set satisfies (NP) or (SP). We finish this section with the following


\begin{remark}\label{thm:ball_pert_Payne}
Let $\Omega \subset \RR^n$ be a domain which can be realised as a one-parameter family of real-analytic perturbations of the ball. Let the unit ball be denoted by $\Omega_0$ 
and $\Omega_1 = \Omega$. Then $\{ t \in [0, 1] : \Omega_t \text{ satisfies 
(SP)} \}$ is an open set.
\end{remark}

\section{Applications: stability on low energy nodal sets}
 
In \cite{MS}, we proved that certain simply-connected perturbations of convex planar domains and perforated domains with sufficiently small perforations satisfy the nodal line conjecture. In what follows, we continue that discussion with several other classes of domains. Moreover, we also look at several other  applications of our stability results and look at certain ``perturbation resistant'' feature of angle estimates of the nodal domains in higher dimensions.

\subsection{Payne property of domains with narrow connector}\label{subsec: dumbbells}
As is already pointed out, by the work in \cite{M}, the strong Payne property is known to hold on convex domains. By further work in \cite{MS}, it is also known to hold on domains obtained from small perturbations of (strictly) convex domains. Somehow a natural approach would be to investigate the validity of the conjecture on domains which are in some sense both ``very far'' from being convex, or being small perturbations thereof. A natural class of such domains would be the so-called dumbbell domains.

We consider dumbbells as constructed in \cite{Ji}. Consider two bounded disjoint open sets $\Omega_1$ and $\Omega_2$ in $\RR^n, n\geq 2$ with smooth boundary 
such that the boundary of each domain has a flat region. More precisely, for some positive constant $\xi>0$,

\begin{equation*}
     \overline{\Omega}_1 \cap \{(x_1, x')\in \RR\times \RR^{n-1}: x_1\geq -1; |x'|<3\xi \}= \{(-1, x')\in \pa \Omega_1: |x'|<3\xi \},
\end{equation*}
and 
\begin{equation*}
    \overline{\Omega}_2 \cap \{(x_1, x')\in \RR\times \RR^{n-1}: x_1\leq 1; |x'|<3\xi \}= \{(1, x')\in \pa \Omega_2: |x'|<3\xi \}. 
\end{equation*}

Let $Q$ be a line segment joining the flat segments (as described above) of $\pa\Omega_1$ and $\pa\Omega_2$. For some small enough fixed $\epsilon>0$, consider the dumbbell domain $\Omega_\epsilon$  obtained by joining $\Omega_1$ and $\Omega_2$ with a connector $Q_\epsilon$ denoted by
$$\Omega_\epsilon:= \Omega_1\cup \Omega_2 \cup Q_\epsilon.$$
Here
$$Q_\epsilon= Q_1(\epsilon)\cup L(\epsilon)\cup Q_2(\epsilon)$$
is given by
\begin{align*}
     Q_1(\epsilon) &= \left\{(x_1, x')\in \RR \times \RR^{n-1}:  -1 \leq x_1 \leq -1+2\epsilon; |x'|<\epsilon \rho\left(\frac{-1-x_1}{\epsilon} \right) \right\},\\
    Q_2(\epsilon)  &= \left\{(x_1, x')\in \RR \times \RR^{n-1}:  1-2\epsilon \leq x_1 \leq 1; |x'|<\epsilon \rho\left(\frac{x_1-1}{\epsilon} \right) \right\},\\
     L(\epsilon) &=  \left\{(x_1, x')\in \RR \times \RR^{n-1}:  -1+2\epsilon \leq x_1 \leq 1-2\epsilon ; |x'|<\epsilon  \right\},
\end{align*}
where $\rho\in C^\infty((-2, 0))\cap C^0((-2, 0])$ is a positive bump function satisfying
$\rho(0)=2$ and $\rho(q)= 1$ for  $q\in(-2, -1)$.
 \begin{figure}[ht]
\centering
\includegraphics[height=4cm]{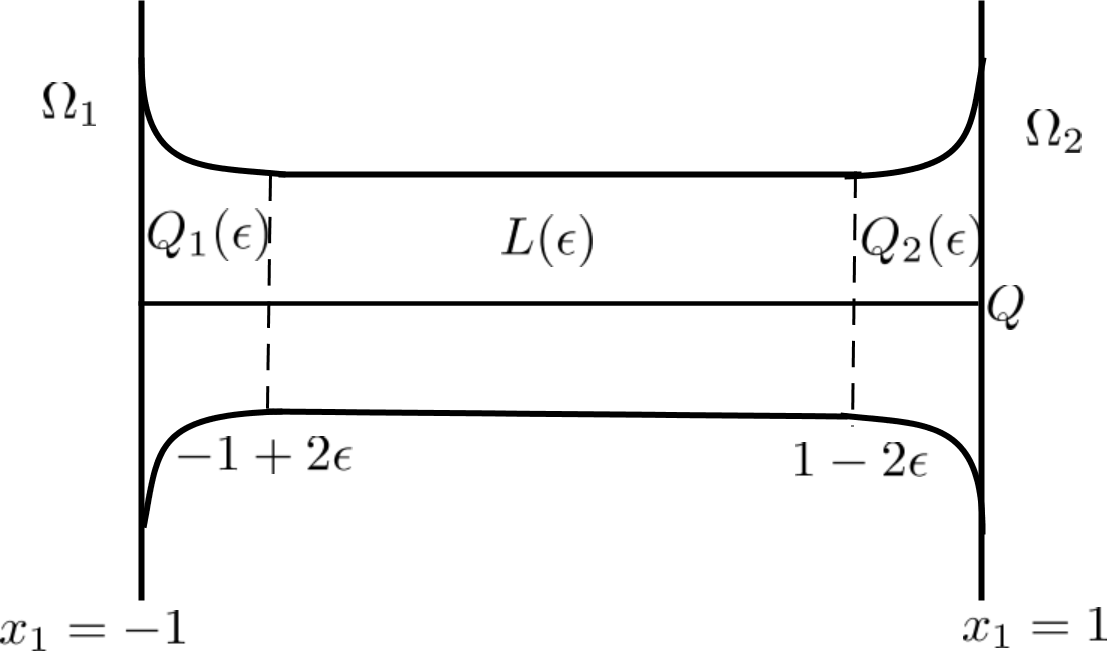}
\caption{Thin connector $Q_\epsilon$}
\end{figure}


Considering the Dirichlet boundary condition on $\Omega_\epsilon$, as pointed out in Chapter 7 of \cite{GN} (see also \cite{D} and Chapter 2 of \cite{He}),
$$\lambda_k(\Omega_\epsilon)\to \lambda_k(\Omega_1\cup \Omega_2) \quad \text{ as } \epsilon\to 0.$$
Here, $\lambda_k(\Omega_1\cup \Omega_2)$ denotes the $k$-th element after rearranging the Dirichlet eigenvalues of $\Omega_1$ and $\Omega_2$ non-decreasingly.  Let $\Lambda_i$  denote the spectrum of $\Omega_i$ ($i=1, 2$),  
and $\varphi_k^{\Omega_i}$ denote the $k$-th Dirichlet eigenfunction of $\Omega_i$ corresponding to the eigenvalue $\lambda_k^{\Omega_i}$. 
 If $\Lambda_1\cap \Lambda_2= \emptyset$ then each eigenfunction $\varphi_{k,\epsilon}$ on the domain $\Omega_\epsilon$ approaches in $L^2$-norm  an eigenfunction $\varphi_{k,0} := \varphi_{k'}^{\Omega_i}$ (for some $i=1,2$ and $k'\leq k$) 
which is fully localised in one subdomain $\Omega_i$ and zero in the other. The fact that the spectra of $\Omega_1$ and $\Omega_2$ do not intersect is important for localisation of the eigenfunctions to exactly one subdomain $\Omega_i$. 

We are interested in looking at the nodal sets of the second eigenfunctions of these dumbbell domains with narrow connectors. 
Without loss of generality, let the eigenfunctions $\varphi_{2, \epsilon_i}$ localise (in the sense described above) on $\Omega_1$ as $\epsilon_i\to 0$. Then, we can rearrange $\Lambda_1 \sqcup \Lambda_2$ in the following two ways.

Case I: $\lambda_1^{\Omega_2}< \lambda_1^{\Omega_1}< \lambda_j^{\Omega_i} \leq \cdots$ for some $i=1, 2$, and  $j \geq 2$;

Case II: $\lambda_1^{\Omega_1}<\lambda_2^{\Omega_1}< \lambda_j^{\Omega_i}\leq \cdots$, for some $i=1, 2$, and $j \in \NN$. 

In general, we label the above arrangement as $\lambda_{1, 0} \leq \lambda_{2, 0} \leq \dots$. Now, 
redefine $\varphi_{2, \epsilon_i}, \varphi_{2,0}$ on $\RR^n$ as
\begin{equation*}
    \varphi_{2, \epsilon_i} = \begin{cases}
        & \varphi_{2, \epsilon_i}  \quad \text{ on } \Omega_{\epsilon_i}, \\
        & 0, \qquad \text{otherwise}.
    \end{cases}
    \quad \text{and} \quad 
    \varphi_{2,0} = \begin{cases}
        & \varphi_1^{\Omega_1}~(\text{or, } \varphi_2^{\Omega_1 } \text{ for Case II}) \quad \text{ on } \Omega_1,\\
        & 0, \qquad\qquad\qquad\qquad\qquad \text{ otherwise}.
    \end{cases}
\end{equation*}

Since we have assumed that the second eigenfunction localises on $\Omega_1$, we have that $$\|\varphi_{2,\epsilon_i}-\varphi_{2,0}\|_{L^2(\RR^n)}\to 0 \text{ as } \epsilon_i\to 0.$$

We now begin proving Theorem \ref{thm: narrow connectors} which deals with Case II.

\begin{proof}

Consider a smooth hypersurface $\Gamma'\subset \overline{\Omega}_1$ such that there exists a $\delta$-tubular neighbourhood of $\Gamma'$ (denoted by $T_{\Gamma',\delta}$) contained in $\Omega_1$ and away from the perturbation. Moreover, $\Gamma'$ is chosen in such a way that 
\begin{itemize}
    \item $\Gamma'$ divides every $\Omega_\epsilon$ into exactly two components with one of the components being $\Omega'\subset \Omega_1$ (see Figure \ref{fig: Construction of the dumbbell domains});
    \item $\pa \Omega_1$ is not a subset of the closure of either components (in other words, we are not considering $\Gamma'$ to be a closed smooth hypersurface completely contained inside $\Omega_1$). 
\end{itemize}

Now, define $\Omega= \Omega'\cup T_{\Gamma', \delta}$, and $\Gamma = \Gamma'+\delta$ (the outer boundary of $T_{\Gamma', \delta}$ with respect to $\Omega'$). Note that the boundary of $\Omega$ can be divided into two parts, namely $\Gamma$ and $\Gamma^*:= \overline{\Omega}\cap {\pa\Omega}_1$. 


We want to prove that $\varphi_{2,\epsilon}\to \varphi_{2,0}$ in $C^0(\Omega')$. We divide $\Omega'$ into two regions, $\Omega_1'$ and $\Omega_2'$, such that $\overline{\Omega'}=\overline{\Omega_1'}\cup \overline{\Omega_2'}$. Here,
$\Omega_1':= \{x\in \Omega: \dist(x, \pa \Omega)>\delta\}$, the ``inner $\delta$-shell of $\Omega_1$, and $\Omega_2':= \Omega'\setminus \overline{\Omega_1'}$ (see Figure \ref{fig: Construction of the dumbbell domains}). 

 \begin{figure}[ht]
\centering
\includegraphics[height=4cm]{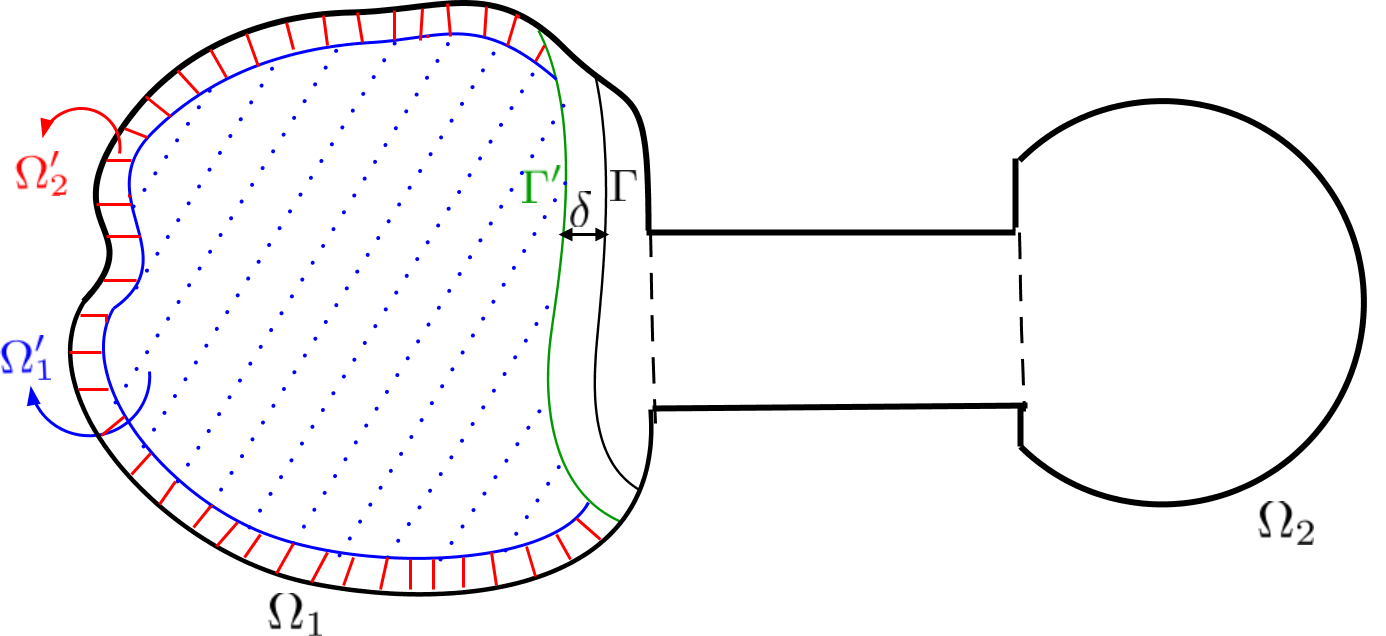}
\caption{$\Omega':=$ interior of $\overline{\Omega_1'}\cup \overline{\Omega_2'}$ }
\label{fig: Construction of the dumbbell domains}
\end{figure}

We first prove that $\varphi_{2,\epsilon}\to \varphi_{2,0}$ in $C^0(\Omega_1')$. Consider 
$$(\Delta+\lambda_{2,0})[\varphi_{2,\epsilon_i}-\varphi_{2,0}]=(\lambda_{2,0}-\lambda_{2,\epsilon_i})\varphi_{2,\epsilon_i} \text{ on } \Omega_1'$$
and 
$$(\Delta+\lambda_{2,\epsilon_i})\varphi_{2,\epsilon_i}=0 \text{ on } \Omega_{\epsilon_i}.$$

\allowdisplaybreaks

Observe that $\Omega_1'$ is compactly contained inside $\Omega \subset\Omega_1$. Now, applying Theorem 8.24 and Theorem 8.15 of \cite{GT} in the above two equations consecutively, we have that for some $q>n$ and $\nu>\sup\{\lambda_{2,0}, \lambda_{2,\epsilon_i}\}$,
\begin{align*}
    \|\varphi_{2,\epsilon_i}-\varphi_{2,0}\|_{L^\infty(\Omega_1')}  &\leq  C\left(\|\varphi_{2,\epsilon_i}-\varphi_{2,0}\|_{L^2(\Omega)}+ \|(\lambda_{2,0}-\lambda_{2,\epsilon_i})\varphi_{2,\epsilon_i} \|_{L^{q/2}(\Omega_1')}\right)\\
    &\leq  C\left(\|\varphi_{2,\epsilon_i}-\varphi_{2,0}\|_{L^2(\Omega)}+ C^*|(\lambda_{2,0}-\lambda_{2,\epsilon_i})|\cdot\|\varphi_{2,\epsilon_i} \|_{L^{\infty}(\Omega_1')}\right)\\
    &\leq  C\left(\|\varphi_{2,\epsilon_i}-\varphi_{2,0}\|_{L^2(\Omega)}+ C^*|(\lambda_{2,0}-\lambda_{2,\epsilon_i})|\cdot\|\varphi_{2,\epsilon_i} \|_{L^{\infty}(\Omega_{\epsilon_i})}\right)\\
    &\leq  C\left(\|\varphi_{2,\epsilon_i}-\varphi_{2,0}\|_{L^2(\Omega)}+ C'|\lambda_{2,0}-\lambda_{2,\epsilon_i}|\cdot\|\varphi_{2,\epsilon_i} \|_{L^{2}(\Omega_{\epsilon_i})}\right)\\
     &\leq  C\left(\|\varphi_{2,\epsilon_i}-\varphi_{2,0}\|_{L^2(\RR^n)}+ C'|\lambda_{2,0}-\lambda_{2,\epsilon_i}|\cdot\|\varphi_{2,\epsilon_i} \|_{L^{2}(\RR^n)}\right),
\end{align*}

where $C, C'$ depends on $q, \nu, \delta$ and $|\Omega_{\epsilon_i}|$. For each $i$, $|\Omega_{\epsilon_i}|$ is uniformly bounded which implies that the constants on the right are independent of $i$. 

Now using $\lambda_{2,\epsilon_i} \to \lambda_{2,0}$ and the fact that $\|\varphi_{2,\epsilon_i}-\varphi_{2,0}\|_{L^2(\RR^n)}\to 0$, we have that as $i\to \infty$ 
\begin{align*}
    \|\varphi_{2,\epsilon_i}-\varphi_{2,0}\|_{L^\infty(\Omega_1')}  \to 0.
\end{align*}

Now, considering the other part $\Omega_2'$, note that $\varphi_{2,\epsilon_i}, \varphi_{2,0}=0$ on $\pa \Omega_2'\cap \pa\Omega_1$. Now, using Theorem 1.1 of \cite{X}, we have that the supremum norm of the gradients of $\varphi_{2,\epsilon_i}, \varphi_{2,0}$ are bounded above uniformly, which in turn implies that $\varphi_{2,\epsilon_i}, \varphi_{2,0}=0$ is sufficiently close to 0. This combined with the above uniform convergence gives us our required $C^0$-convergence on $\Omega'$.


From our assumption, we have $\lambda_{2,0}=\lambda_2^{\Omega_1}$ with $\varphi_{2,\epsilon}|_{\Omega_1}\to \varphi_2^{\Omega_1}$ in $L^2(\Omega_1)$. 
Additionally, choose the hypersurface $\Gamma'\subset \Omega_1$ such that $\NNN(\varphi_2^{\Omega_1})$ lies in $\Omega'$. 

 If possible, let $\NNN(\varphi_{2,\epsilon_i})$ intersect the connector $Q_{\epsilon_i}$ for every  $\epsilon_i>0$, 
and $\NNN(\varphi_{2,\epsilon_i}) \cap \Gamma' = \{p_i\}\in \Omega_1$. Then there exists a subsequence $\{p_j\}\subset \{p_i\}$ and some $p\in \Omega_1$ such that $\{p_j\}\to p$, and from Lemma \ref{lem:lim_M_1}, we have that $p\in \NNN(\varphi_2^{\Omega_1})$. Recall that we have assumed that $Q_\epsilon$ is away from $\NNN(\varphi_2^{\Omega_i})$ ($i=1,2$), that is, $Q_\epsilon$ does not intersect $\NNN(\varphi_2^{\Omega_1})$ for any $\epsilon>0$. This leads to a contradiction which 
implies that, there exists $\epsilon_0>0$ such that $\NNN({\varphi_{2,\epsilon}}) \subset \Omega'$ for any $\epsilon< \epsilon_0$. 

 \begin{figure}[ht]
\centering
\includegraphics[height=4cm]{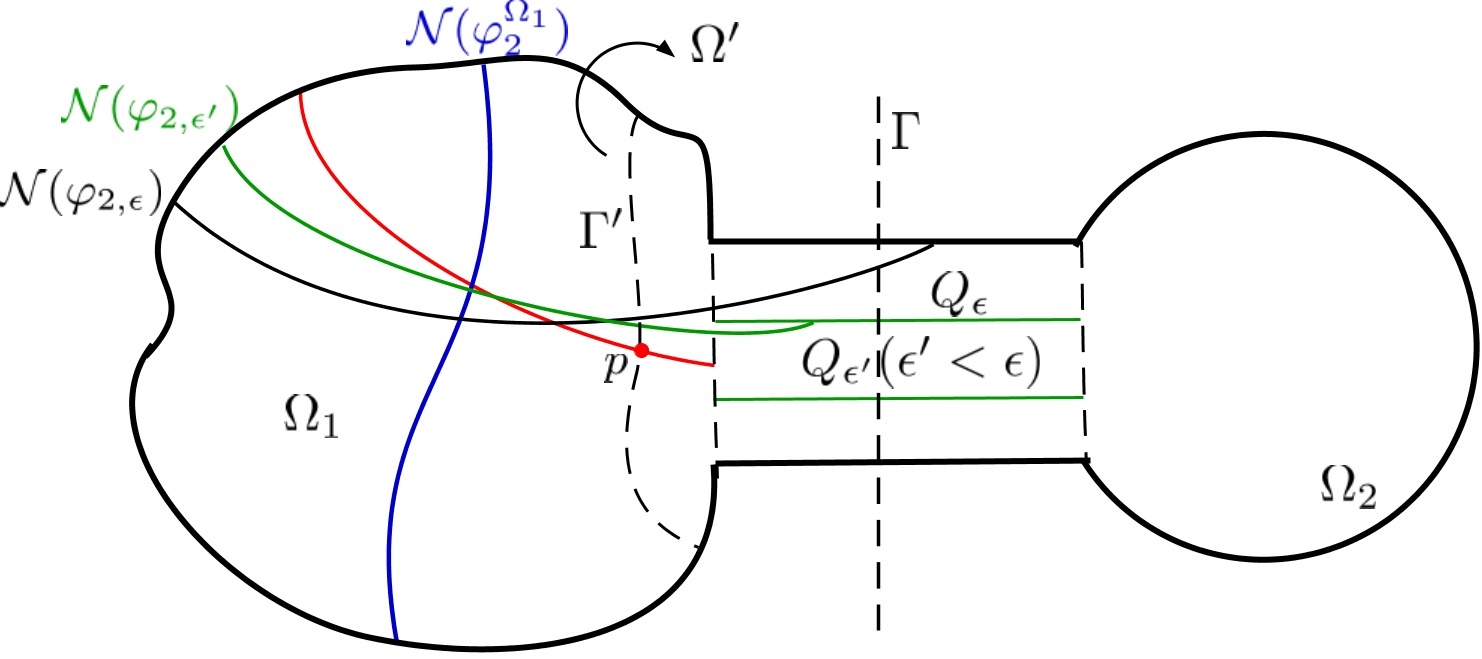}
\caption{Behaviour of nodal line as $\epsilon \to 0$}
\label{fig: Behaviour of nodal line}
\end{figure}

Now, using Theorem \ref{thm: Henrot_conv_theo} one can extract a subsequence called  $\NNN_{\varphi_{2,\epsilon_i}}$ ($\epsilon_i< \epsilon_0$) which converges to a set $X \subset \Omega_1$ in the Hausdorff metric and by Lemma \ref{lem:lim_M_1} we know that $X\subset \NNN(\varphi_2^{\Omega_1})$. 
Now following the argument as in Proposition \ref{prop: SP/NP open condition}, we conclude the proof. 
\end{proof}

\begin{remark}
With some obvious modifications, the above proof also tells us that if $\Omega_1$ satisfies (NP), then so does $\Omega_\epsilon$ for sufficiently small $\epsilon$.
\end{remark}

So far, we have only looked at the nodal sets of the dumbbells under Case II. Now we turn our attention to the location of nodal sets under Case I.

\begin{theorem}
    Consider a family of dumbbells as described in Theorem \ref{thm: narrow connectors} with the condition that $\lambda_{2,\epsilon}\to \lambda_1^{\Omega_1}$. Then, for sufficiently small $\epsilon$, $\NNN(\varphi_{2,\epsilon})$ does not  enter $\Omega'\subset \Omega_1$. 
\end{theorem}
We would like to point out that the only restriction on the choice of $\Gamma'$ in this case is that $\Gamma'$ divides every $\Omega_{\epsilon}$ into exactly two components as described in the beginning of the proof of Theorem \ref{thm: narrow connectors}. In other words, one can choose $\Gamma'\subset\Omega_1$ sufficiently close to the connectors.
\begin{proof}
Considering the case when $\lambda_{2,0}=\lambda_1^{\Omega_1}$ and $\displaystyle \varphi_{2,\epsilon_i}\to \varphi_1^{\Omega_1}$ in $L^2(\Omega_1)$, 
note that $\varphi_1^{\Omega_1}$ does not change sign in $\Omega_1$. Without loss of generality, assume that $\varphi^{\Omega_1}_1>0$  in $\Omega_1$. 
Moreover, we know that $\NNN(\varphi_{2,\epsilon})$ divides $\Omega_\epsilon$ into two components. If possible, let $\NNN(\varphi_{2,\epsilon})$ enters $\Omega'$ for every $\epsilon>0$. Then $\NNN(\varphi_{2,\epsilon})$ intersects $\Gamma'$ for every $\epsilon$. Since $\varphi^{\Omega_1}_1>0$ in $\Omega_1$ and $\displaystyle \|\varphi_{2,\epsilon}- \varphi_1^{\Omega_1}\|_{C^0(\overline{\Omega'})} \to 0$, it is clear that for small enough $\epsilon$, $\varphi_{2,\epsilon}>0$ on $\Gamma'$, which is a contradiction. 
\end{proof}
 
A natural follow-up to the above theorem is that, under Case I, {\em can we say that for sufficiently small $\epsilon$, $\NNN(\varphi_{2,\epsilon})$ does not enter $\Omega_1$ at all?}

We comment in passing that dumbbell domains have many interesting properties that are of interest to spectral theorists. For example, they provide examples of domains which have arbitrarily low second Neumann eigenvalue, and almost satisfy a weak version of the hot spot conjecture (see \cite{MS1}). There is also a significant literature on mass concentration questions in the connector of the dumbbell, for example see \cite{BD, vdBB, GM3, DNG}.  

\subsection{Counterexample of Payne property in higher dimensions}\label{subsec:counter_Four}

We begin by discussing the planar counter-example as given in \cite{HHN}. First we choose two concentric balls $B_{R_1}$ and $B_{R_2}$ in $\RR^2$ such that 
$$\lambda_1(B_{R_1}) < \lambda_1(B_{R_2}\setminus \overline{B_{R_1}}) < \lambda_2(B_{R_1}).$$
Next, we carve out holes into $\pa B_{R_1}$. 
Let $N\in \NN$ and $\epsilon< \pi/N$. The domain $\Omega_{N, \epsilon}$ is defined as
\begin{equation}
    \Omega_{N, \epsilon}= B_{R_1} \cup (B_{R_2}\setminus \overline{B_{R_1}}) \cup \left(\cup_{j=0}^{N-1}\left\{x\in \RR^2: r= R_1, \omega\in \left(\frac{2\pi j}{N}-\epsilon, \frac{2\pi j}{N}+\epsilon\right) \right\} \right).
\end{equation}
Then the first nodal line does not intersect the boundary for sufficiently large $N$ and small $\epsilon$.
\vspace{2mm}
 \begin{figure}[ht]
\centering
\includegraphics[scale=0.8]{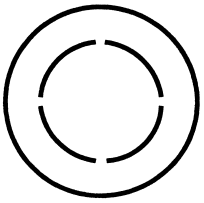}
\caption{N=4 (\cite{HHN}, Figure 1)}
\label{fig:HHN counterexample domain}
\end{figure}

In higher dimensions ($n\geq 3$), the domain constructed by Fournais was motivated from the example in \cite{HHN} described above and defined as follows:
\begin{equation}
    \Omega_\epsilon= B_{R_1} \cup \left(B_{R_2}\setminus \overline{B_{R_1}}\right) \cup \left(\bigcup_{i=1}^N B(x_i, \epsilon) \right),
\end{equation}
where $B_R$ is a ball of radius $R$ centered at 0, and $x_1,  \cdots, x_N\in S^{n-1}_{R_1}$ are chosen in such a way that the ``patches'' $B(x_i, \epsilon) \cap S^{n-1}_{R_1}$ are evenly distributed over $S^{n-1}_{R_1}$, the sphere with center at $0$ and  radius $R_1$. For convenience, moving forward we will refer to the sphere $S^{n-1}_{R_1}$ as $S$. Also, $R_1$ and $R_2$ are chosen such that 
$$\lambda_1(B_{R_1}) < \lambda_1(B_{R_2}\setminus \overline{B_{R_1}}) < \lambda_2(B_{R_1}).$$
Then for small enough $\epsilon$ the second eigenfunction $\varphi_{2,\epsilon}$ satisfies
\begin{equation*}
    \NNN(\varphi_{2,\epsilon})\cap \pa \Omega_\epsilon=\emptyset.
\end{equation*}
The main idea in \cite{F} is to prove that for small enough $\delta>0$, there is $\epsilon>0$ such that $\varphi_{2,\epsilon}(x)>0$ on $|x|=R_1-\delta$. Then using various assumptions made during the construction along with certain topological restrictions of the first nodal set $\NNN(\varphi_{2,\epsilon})$, one concludes that $\NNN(\varphi_{2,\epsilon})$ is contained inside $B_{R_1-\delta}$.

Note that the domain $\Omega_\epsilon$ described above is not simply connected. Also, $\Omega_\epsilon$ has a simple spectrum. Our goal in this section is to produce a simply connected domain whose nodal set does not intersect the boundary. 

Let $\Omega_0:= \Omega_\epsilon$, where $\Omega_\epsilon$ is the above described domain of Fournais for which the nodal set is contained in $B_{R_1-\delta}$. Throughout the rest of the proof, the above $\epsilon$ and $\delta$ will remain fixed. From $\Omega_0$ we can construct simply connected domains by adding $(n -  1)$-dimensional ``tunnels'' or ``strips'' $T_\eta$ along $S$ in between the ``patches'' $B(x_i, \epsilon)$ such that every patch 
is connected to the neighbouring patches by tunnels (see  Figure \ref{fig:Topologically simple counterexample to Payne}). Our idea is to make these tunnels narrow enough so that the nodal set of $\Omega_0$ does not get sufficiently perturbed. 

 \begin{figure}[ht]
\centering
\includegraphics[scale=0.11]{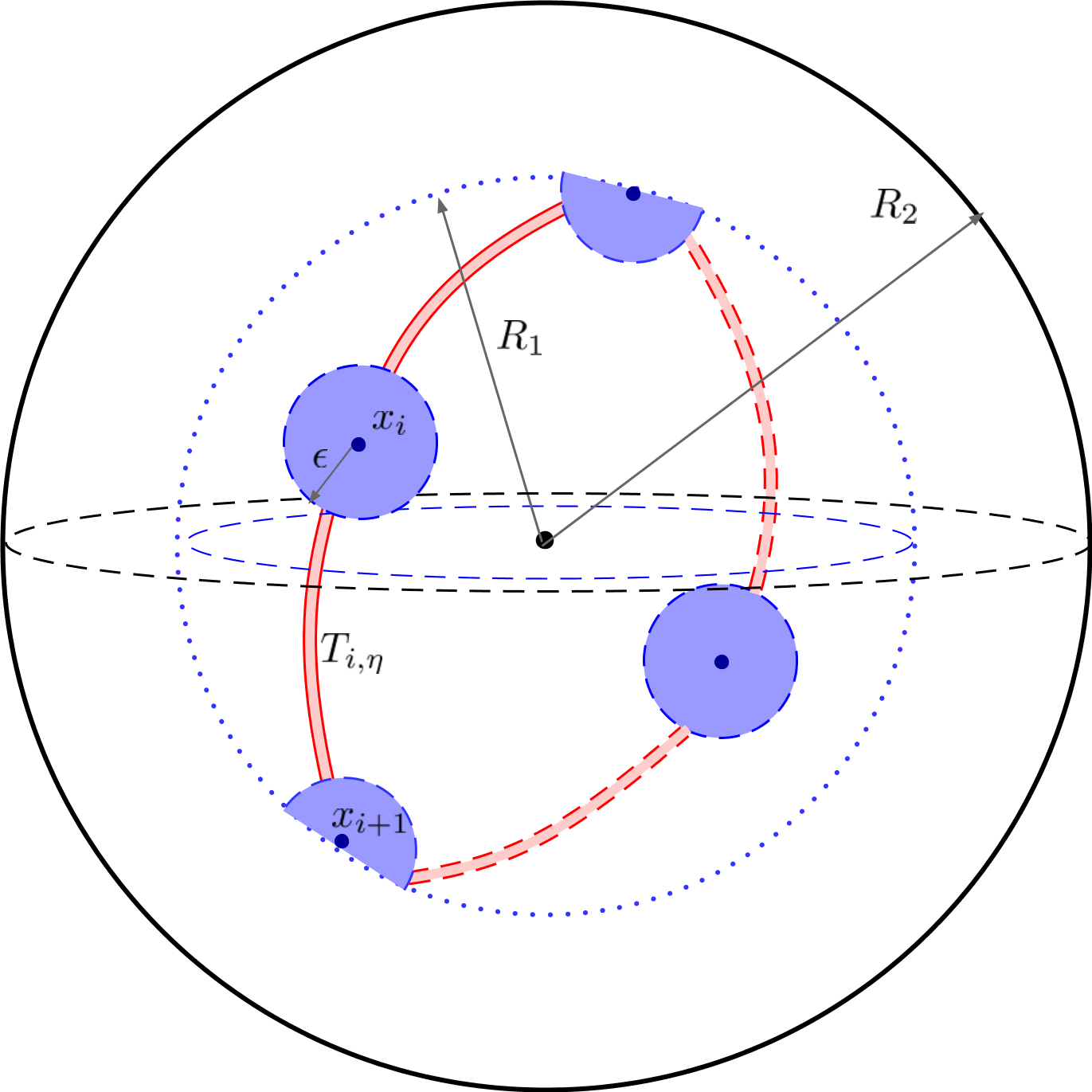}
\caption{Topologically simple counterexample to Payne}
\label{fig:Topologically simple counterexample to Payne}
\end{figure}

Let $\eta>0$. For any $i, j\in \{1, \cdots, N\}~(i\neq j)$ let $p_{ij}(t):[0,1]\to S$ be a path between $x_i$ and $x_j$ along $S$ such that the length of $p_{ij}$ is $\dist_S(x_i, x_j)$, the geodesic distance between $x_i$ and $x_j$ on S. Let $t_0$ and $t_1\in [0,1]$ be such that 
$p_{ij}(t_0) \in \pa B(x_i, \epsilon)$ and 
$p_{ij}(t_1) \in \pa B(x_j, \epsilon)$. Now consider the path segment $P_{ij}=[p_{ij}(t_0), p_{ij}(t_1)]$. Let $\tau_{ij}(\eta)$ denote the $\eta-$tubular neighbourhood of $P_{ij}$. Define the tunnel $T_\eta^{i,j}:= S\cap \tau_{ij}(\eta)$. 
Let there be $k_N$ tunnels in total. 
Denote
$$T_\eta:= \bigcup_{i=1}^{k_N} T^{i, j}_{\eta}.$$
Now, we define a family of domains $\Omega_\eta$ as
$$\Omega_\eta:= B_{R_1} \cup \left(B_{R_2}\setminus \overline{B_{R_1}}\right) \cup \left(\cup_{i=1}^N B(x_i, \epsilon) \right) \cup T_\eta = \Omega_0\cup T_\eta. $$

It is easy to check that  $\Omega_\eta$ is simply connected and given any sequence $\eta_{q}\searrow 0$ there is a subsequence $\{\eta_p\}\subseteq \{\eta_q\}$ such that $\Omega_{\eta_p}$ converges to $\Omega_0$ in Hausdorff metric. 
Now we check that the perturbation to the nodal set is controlled. 

Let $\varphi_{j,\eta}$, $\varphi_{j, 0}$ denote the eigenfunction corresponding to eigenvalues $\lambda_{j, \eta}$, $\lambda_{j, 0}$ of the Dirichlet-Laplacian $-\Delta_\eta$, $-\Delta_0$. We assume that the eigenfunctions are $L^2$-normalized. Also note that, from our assumption, we have that $\varphi_{2,0}>0$ in $\Omega_0\setminus B_{R_1-\delta}$.

Let $\{\eta_p\}\searrow 0$ be any strictly monotonically decreasing sequence and $X_{p}:= B_{R_2}\setminus \Omega_{\eta_p}$. Note that $\{X_p\}$ is a increasing family of compact sets.
Define
$$P_p:= \bigcup_{k\geq p}^\infty X_k \quad \text{and} \quad Q_p:= \bigcap_{k\geq p}^\infty X_k.$$
Using the convention from \cite{Sto}, we have that $P_p\nearrow X= \underline{\lim }X_p$ and $Q_p\searrow X= \overline{\lim }X_p$ where $X= B_{R_2}\setminus \Omega_0'$ and $\Omega_0':= \Omega_0\cup (\cup P_{ij})$ . Also, for any $p, m\in \NN$, $X_p \bigtriangleup X_m \subset T_{\eta_1}$ which has finite capacity and $\capacity(\Omega\bigtriangleup \Omega')=0$. Then using Theorem 2.2 of \cite{Sto}, we have that, $-\Delta_{\eta_p}$ converges to $-\Delta_0$ as $p\to \infty$ in norm resolvent sense (recall that if $\{T_p\}_{n=1}^\infty$ and $T$ are unbounded self-adjoint operators, then $T_p\to T$ in norm resolvent sense means that for some $z\in \mathbb{C} \setminus \mathbb{R}$, $\|(zI- T_p)^{-1}- (zI- T)^{-1}  \|\to 0$ as $p\to \infty$). In particular, for $\lambda_{2,0}$ there exists a sequence $\eta_p\to 0$ such that
$$\lambda_{2, \eta_p}\to \lambda_{2,0}.$$ Redefining $\varphi_{2,\eta_p}, \varphi_{2,0}$ by $0$ on $\RR^n\setminus \Omega_{\eta_p}, \RR^n\setminus \Omega_0$ we also have that $\varphi_{2,\eta_p}\to \varphi_{2,0}$  in $L^2(\RR^n)$. 

We know that $\NNN(\varphi_{2,0})$ is completely contained inside $B_{R_1-\delta}$. 
Now we would like to show that $\varphi_{2, \eta_p}\to \varphi_{2,0}$ in $C^0(B_{R_1-\delta})$. 

Consider 
$$(\Delta+\lambda_{2,0})[\varphi_{2,\eta_p}-\varphi_{2,0}]=(\lambda_{2,0}-\lambda_{2,\eta_p})\varphi_{2,\eta_p} \text{ on } B_{R_1-\delta'},$$
and 
$$(\Delta+\lambda_{2,\eta_p})\varphi_{2, \eta_p}=0 \text{ on } \Omega_{\eta_p}$$
where  $0<\delta'<\delta$. 
Now using Theorem 8.24 and Theorem 8.15 of \cite{GT} consecutively on the above equations as done in the proof of Theorem \ref{thm: narrow connectors}  we have that for some $q>n$ and $\nu>\sup\{\lambda_{2,0}, \lambda_{2,\eta_p}\}$,
\begin{equation}
    \|\varphi_{2,\eta_p}-\varphi_{2,0}\|_{L^\infty(B_{R_1-\delta})} \leq  C\left(\|\varphi_{2,\eta_p}-\varphi_{2,0}\|_{L^2(\RR^n)}+ C'|\lambda_{2,0}-\lambda_{2,\eta_p}|\cdot\|\varphi_{2,\eta_p} \|_{L^{2}(\RR^n)}\right),
\end{equation}

where $C, C'$ depends on $n, q, \nu,$ and $|\Omega_{\eta_p}|$. For each $p$, $|\Omega_{\eta_p}|$ is uniformly bounded which implies that the constants on the right hand are independent of $p$. Now using $\lambda_{2,\eta_p} \to \lambda_{2,0}$ and $\|\varphi_{2,\eta_p}-\varphi_{2,0}\|_{L^2(\RR^n)}\to 0$ we have that as $p\to \infty$ 
\begin{align*}
    \|\varphi_{2,\eta_p}-\varphi_{2,0}\|_{L^\infty(B_{R_1-\delta})}  \to 0.
\end{align*}
which gives our desired $C^0(B_{R_1-\delta})$ convergence.

Finally, using Lemma \ref{lem:lim_M_1} and Proposition \ref{prop: SP/NP open condition}, we know that $\NNN(\varphi_{2, \eta_n})$ converges to $\NNN(\varphi_{2,0})$. So, for sufficiently large $n_0\in \NN$ we have, 
$$\NNN(\varphi_{2, \eta_{n_0}}) \subset \subset B_{R_1-\delta}.$$
In other words, we have a simply connected domain $\Omega_{\eta_{n_0}}\in \RR^n (n\geq 3)$ for which  $$\NNN(\varphi_{2, \eta_n})\cap \pa \Omega_{\eta_{n_0}} = \emptyset.$$

Now we look at the proof of Theorem \ref{thm: satisfying NP} and the proof mimics the construction used in \cite{MS}. We give a brief sketch of the proof and leave out the details to avoid repetition. 

 \begin{figure}[ht]
\centering
\includegraphics[scale=0.11]{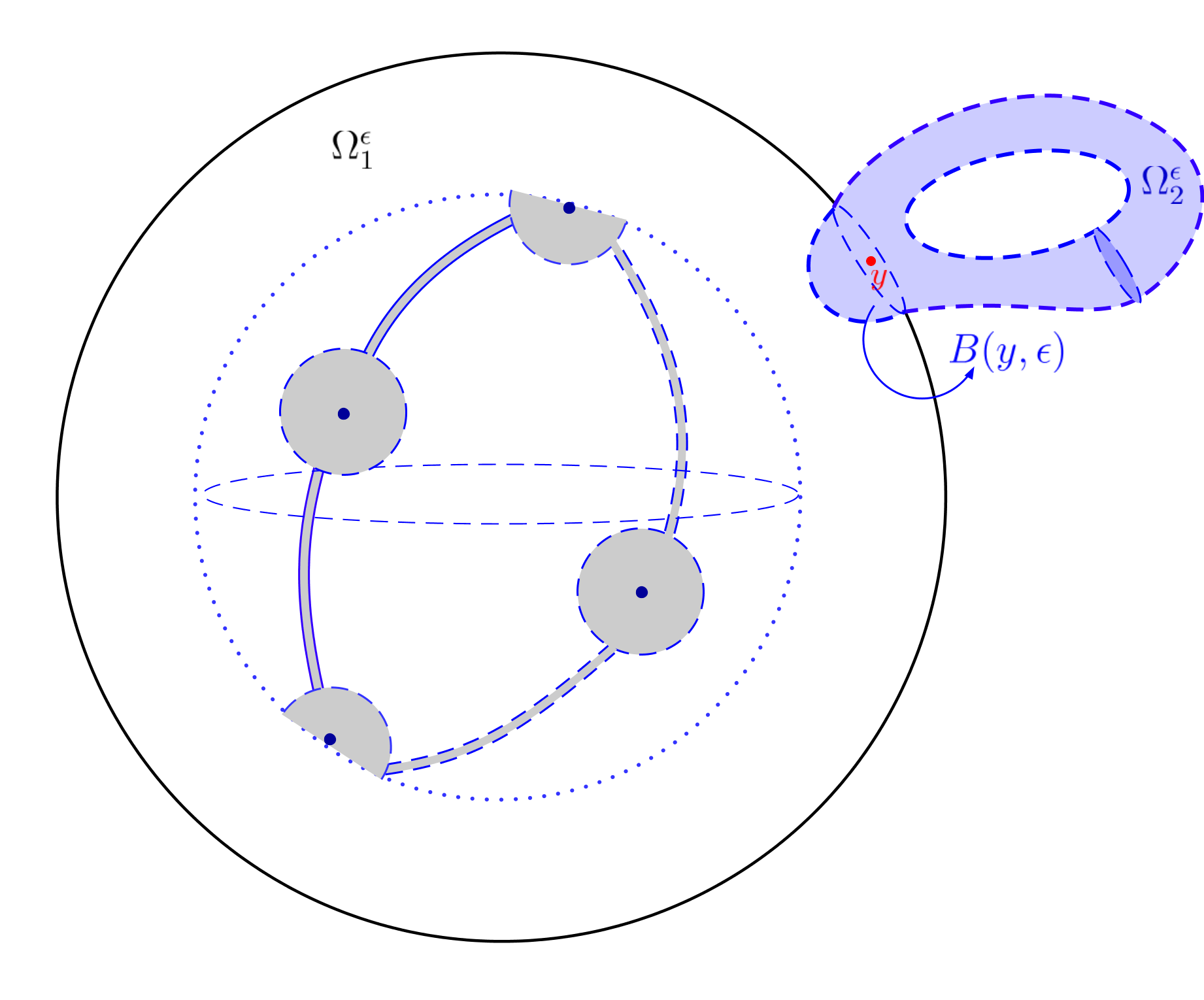}
\caption{Counterexample of Payne property with any prescribed topology}
\label{fig:Counterexample with any prescribed topology}
\end{figure}

\begin{proof}
Given $\Omega\subset \RR^n$ ($n\geq 3$), the way to construct the required domains is by taking the connected sum of $\Omega_{\eta_{n_0}}$ with the given domain $\Omega$. But we note that one cannot blindly attach one domain with another as described above since the attached metric will not be Euclidean (or even flat) in general. So the family of metrics is to be designed precisely to ensure that each deformation $\Omega_\epsilon= \Omega_{\eta_{n_0}} \#_\epsilon \Omega$ is a Euclidean domain.

Consider a one-parameter family of deformations $\Omega_\epsilon$, where $\Omega_\epsilon$ can be written as a disjoint union $\Omega_1^\epsilon \sqcup \Omega_2^\epsilon$ (see Figure \ref{fig:Counterexample with any prescribed topology}), where 
\begin{itemize}
    \item $\Omega_1^\epsilon := \Omega_{\eta_{n_0}} \setminus B(\xi, \epsilon)$, for some $\xi \in \pa B_{R_2}$.
    \item $\Omega_2 \subset \RR^n$ is another domain defined as $\Omega\cup B(y, 1)$ with $y\in \pa \Omega$, and $\Omega_2^\epsilon$ is obtained from $\Omega_2$ by scaling $g|_{\Omega_2^\epsilon} = \epsilon^2 g|_{\Omega_2}$.
\end{itemize}

Then, from \cite{MS}, we have that for sufficiently small $\epsilon$, $\Omega$ satisfies (NP). Moreover, since $\Omega_{\eta_{n_0}}$ is simply connected, we have that $\Omega_1^\epsilon$ is simply connected which implies that $\Omega_\epsilon$ has the prescribed topology of $\Omega$.
\end{proof}

\subsection{Angle estimates of nodal sets} \label{sec:geometric_prop}
A consequence of the fact that (SP) is open (as proved in Proposition \ref{prop: SP/NP open condition}) is that the angle at the nodal critical points of the first nodal set with order of vanishing $2$ and satisfying  (SP)  remains stable under perturbation.

\subsubsection{Opening angles nodal domains in the interior and the boundary} It was shown by Melas \cite{M} that the nodal domain for the second Dirichlet eigenfunction which intersects the boundary $\pa\Omega$ cannot have an ``opening angle'' of $0$ or $\pi$ at the point of intersection. Here, we provide a generalisation of this result from a different perspective, one that was introduced in \cite{GM1}. 


{\em Interior cone conditions.} In dimension $n = 2$, a well-known result of Cheng \cite{Ch} says the following (see also ~\cite{St} for a proof using Brownian motion):

\begin{theorem}
	For a compact Riemannian surface $M$, the nodal set $\NNN_{\varphi_\lambda}$ satisfies an interior cone condition with opening angle $\alpha \gtrsim \frac{1}{\sqrt{\lambda}}$.
\end{theorem}

Furthermore, in dimension $2$, the nodal lines form an equiangular system at a singular point of the nodal set. The idea behind Cheng's proof is the following: using a local power series expansion due to Bers (see Theorem \ref{Zildo} below), near any point of vanishing the eigenfunction ``looks like'' a homogeneous harmonic polynomial whose degree matches the order of vanishing at that point. If the order of vanishing is $k$, then  in two dimensions such a function would be a linear combination of $r^k\cos k\theta$ and $r^k \sin k\theta$. This gives an equiangular nodal junction. 

 \begin{figure}[ht]
\centering
\includegraphics[scale=0.18]{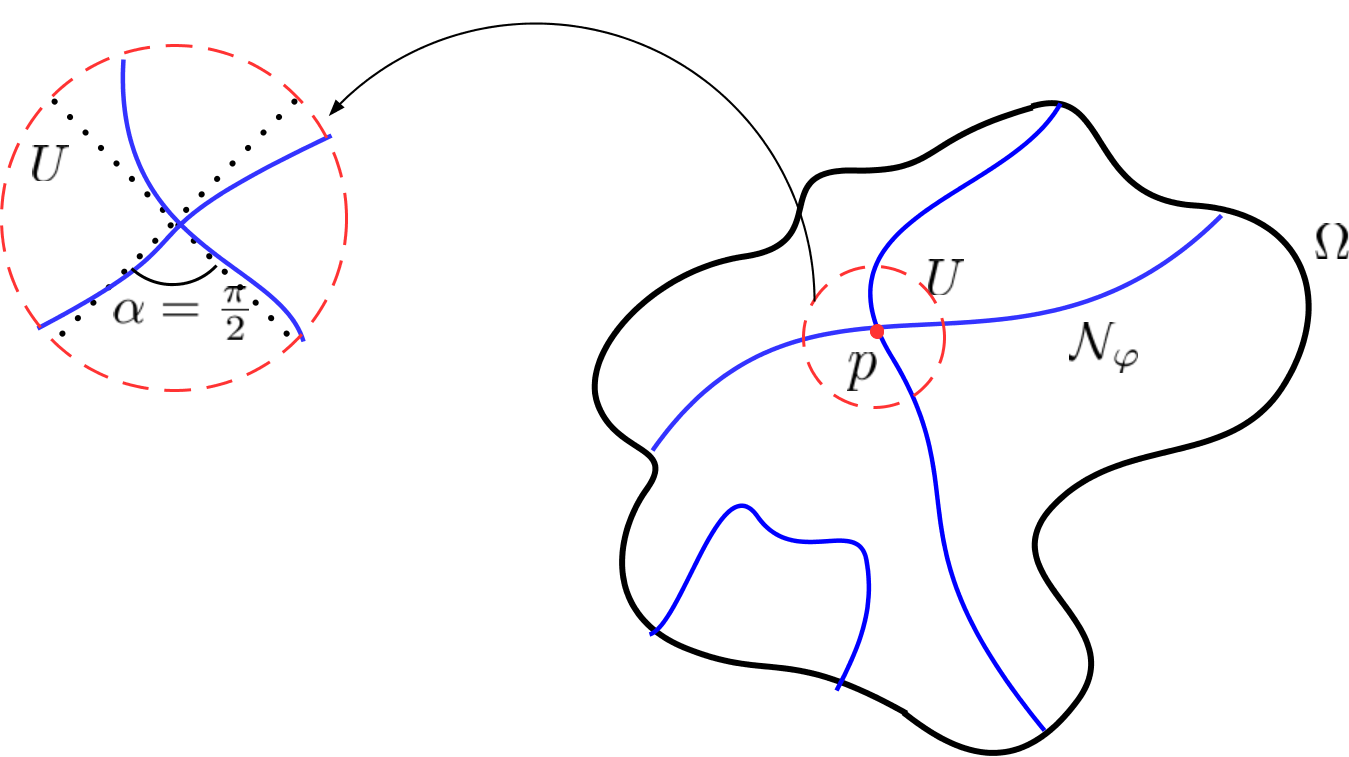}
\caption{Four equiangular ``rays'' from $p$}
\end{figure}
     

The situation is significantly more complicated in higher dimensions. Setting $\dim M \geq 3$, we discuss the question whether at the singular points of the nodal set $\NNN_\varphi$, the nodal set can have arbitrarily small opening angles, or even ``cusp''-like situations, or the nodal set has to self-intersect ``sufficiently transversally''.
We observe that in dimension $n \geq 3$ the nodal sets satisfies an appropriate ``interior cone condition'', and give an estimate on the opening angle of such a cone in terms of the eigenvalue $\lambda$.

Now, in order to properly state or interpret such a result, one needs to define the concept of ``opening angle'' in dimension $n \geq 3$. We start by defining precisely the notion of tangent directions in our setting.
\begin{defi}
	Let $\Omega_\lambda$ be a nodal domain and $x \in \partial \Omega_\lambda$, which means that $\varphi_\lambda(x) = 0$. Consider a sequence $x_n \in \NNN_\varphi$ such that $x_n \to x$. Let us assume that in normal coordinates around $x$, $x_n = \exp(r_nv_n)$, where $r_n$ are non-negative real numbers, and $v_n \in S(T_xM)$, the unit sphere in $T_xM$. Then, we define the space of tangent directions at $x$, denoted by $\mathcal{S}_x\NNN_\varphi$ as
	\begin{equation}
	\mathcal{S}_x\NNN_\varphi = \{v \in S(T_xM) : v = \lim v_n, \text{  where  }x_n \in \NNN_\varphi, x_n \to x\}.
	\end{equation}
\end{defi}

Observe that there are more well-studied variants of the above definition,
for example, as due to Clarke or Bouligand (for more details, see ~\cite{R}). With that in place, we now give the following definition of ``opening angle''.
\begin{defi}\label{def:int_cone_cond}
	
	We say that the nodal domain $\Omega_\lambda$ satisfies an interior cone condition with opening angle $\alpha$ at $x \in \NNN_\varphi\subset\pa \Omega_{\lambda}$, if any connected component of $S(T_xM) \setminus \mathcal{S}_x\pa \Omega_\varphi$ has an inscribed ball of radius $\gtrsim \alpha$.
\end{defi}

We will use Bers scaling of eigenfunctions near zeros (see ~\cite{Be}). We quote the version as appeared in ~\cite{Z}, Section 3.11.
\begin{theorem}[Bers]\label{Zildo}
Assume that $\varphi_\lambda$ vanishes to order $k$ at $x_0$. Let $\varphi_\lambda(x) = \varphi_k(x) + \varphi_{k + 1}(x) + ..... $ denote the Taylor expansion of $\varphi_\lambda$ into homogeneous terms in normal coordinates $x$ centered at $x_0$. Then $\varphi_\kappa(x)$ is a Euclidean harmonic homogeneous polynomial of degree $k$.
\end{theorem}

We also use the following inradius estimate for real analytic metrics (see ~\cite{G}).
\begin{theorem}\label{Zeldiev}
	Let $(M, g)$ be a real-analytic closed manifold of dimension at
	least $3$. If $\Omega_\lambda$ is a nodal domain corresponding to the eigenfunction $\varphi_\lambda$, 
	then there exist constants $\lambda_0, c_1$ and $c_2$ which depend only
	on $(M, g)$, such that
	\beq\label{ra_inrad}
	\frac{c_1}{\lambda} \leq \inrad (\Omega_\lambda) \leq  \frac{c_2}{\sqrt{\lambda}}, \lambda \geq \lambda_0.\eeq
\end{theorem}

Since the statement of Theorem \ref{Zeldiev} is asymptotic in nature, we need to justify that if $\lambda < \lambda_0$, a nodal domain corresponding to $\lambda$ will still satisfy $\inrad (\Omega_\lambda) \geq \frac{c_3}{\lambda}$ for some constant $c_3$. This follows from the inradius estimates of Mangoubi in \cite{Ma}, which hold for all frequencies. Consequently, we can assume that every nodal domain $\Omega$ on $S^n$ corresponding to the spherical harmonic $\varphi_k(x)$, as in Theorem \ref{Zildo}, has inradius $\gtrsim \frac{1}{\lambda}$.

Now we start proving Theorem \ref{thm:ICC}.
\begin{proof}
    Since the eigenequation $-\Delta \varphi_\lambda  = \lambda \varphi_\lambda$ is satisfied at $p$, one can check that the proof of Theorem \ref{Zildo} above still works at $p$, for all $p \in M \cup \pa M$. We observe that Theorem \ref{Zeldiev} applies to spherical harmonics, and in particular the function $\text{exp}^*(\varphi_k)$, restricted to $S(T_{x_0}M)$, where $\varphi_k(x)$ is the homogeneous harmonic polynomial given by expanding $\varphi$ at $p$ in terms of $x \in M \cup \pa M$ given by Theorem \ref{Zildo}. 
	 Also, a nodal domain for any spherical harmonic on $S^2$ (respectively, $S^3$) corresponding to eigenvalue $\lambda$ has inradius $\sim \frac{1}{\sqrt{\lambda}}$ (respectively, $\gtrsim \frac{1}{\lambda^{7/8}}$).
	
	With that in place, it suffices to prove that
	\beq\label{suff}
	\Cal{S}_{x_0}\NNN_\varphi \subseteq \Cal{S}_{x_0}\NNN_{\varphi_k}.
	\eeq
Now by definition, $v \in \Cal{S}_{x_0}N_\varphi$ if there exists a sequence $x_n \in N_\varphi$ such that $x_n \to x_0$, $x_n = \text{exp} (r_nv_n)$, where $r_n$ are positive real numbers and $v_n \in S(T_{x_0}M)$, and $v_n \to v$.

 This gives us,
\begin{align*}
0 & = \varphi_\lambda(x_n)  = \varphi_\lambda(r_n\text{exp }v_n) \\
& = r_n^k\varphi_k(\text{exp }v_n) + \sum_{m > k}r_n^{m}\varphi_m(\text{exp }v_n)\\
& = \varphi_k(\text{exp }v_n) + \sum_{m > k}r_n^{m - k}\varphi_m(\text{exp }v_n)\\
& \to \varphi_k(\text{exp }v), \text{ as } n \to \infty.
\end{align*}

 Observing that $\varphi_k(x)$ is homogeneous, this proves (\ref{suff}).
 \end{proof}
 
Observe that Theorem \ref{thm:ICC} above tells us that the following two situations in Figure \ref{fig:impermissible angle} can never happen at the boundary for the nodal set of any eigenfunction (there is nothing specific about the second eigenfunction).  

\begin{figure}[ht]
    \centering
    \subfloat[\centering Opening angle=$0$]{{\includegraphics[width=4.2cm]{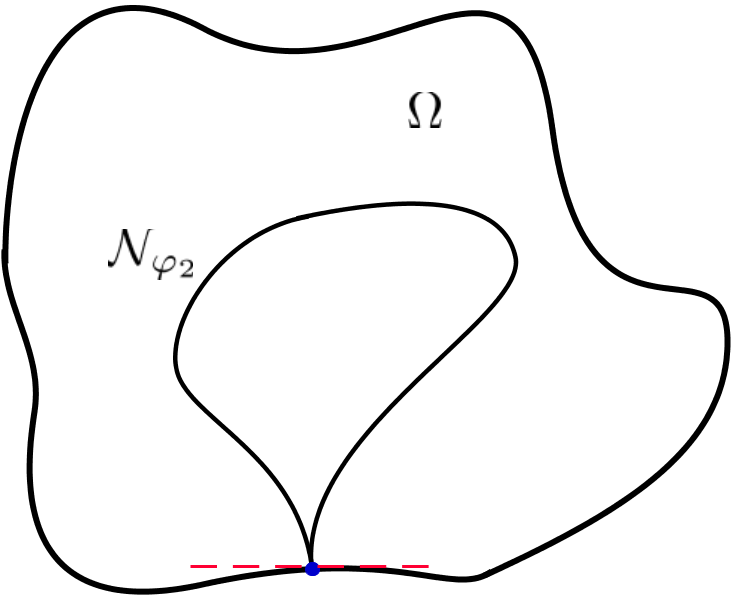} }}
    \qquad \qquad \qquad
    \subfloat[\centering Opening angle=$\pi$]{{\includegraphics[width=4.2cm]{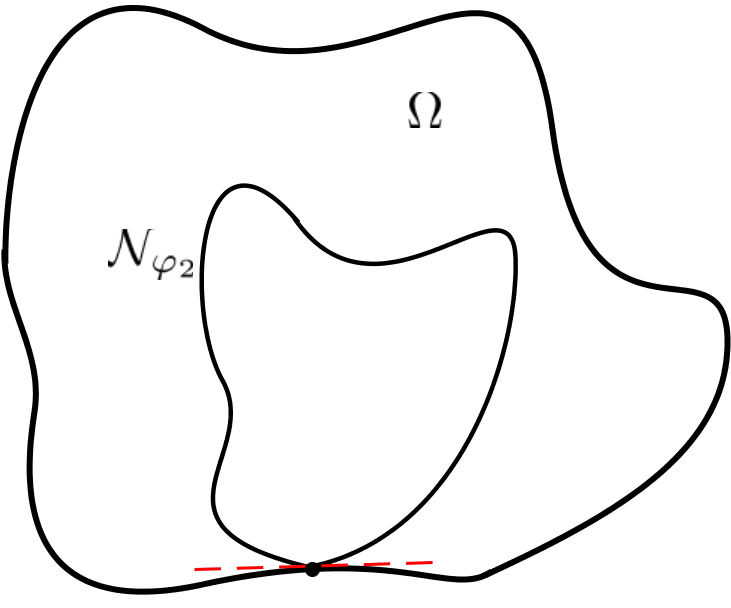} }}
    \caption{Impermissible angle of intersections for any bounded domain in $\RR^2$}
    \label{fig:impermissible angle}
\end{figure}

\begin{remark}\label{Cheng on boundary}
In $\dim M=2$, since any point $p\in \pa  M$ satisfies the eigenequation $-\Delta \varphi=\lambda\varphi$, the local expansion of Bers is true  on the boundary as well. Then using the above ideas of Cheng on the boundary, we have that if $p\in \pa M$ has $k^{th}$ order of vanishing then $\NNN_{\varphi}$ forms an equiangular junction at $p$ with respect to the tangent at $p$.

\end{remark}

\subsubsection{More precise estimates on opening angles}
We 
are now going to investigate in more detail the angle between two nodal hypersurfaces at a point of intersection. In some sense, our results here are going to be higher dimensional analogues of Cheng's result outlined in Remark \ref{Cheng on boundary}.

Very interestingly, such problems have been investigated from a completely different viewpoint in  classical Fourier analysis, namely, the existence of Heisenberg uniqueness pairs. For the sake of completeness, we include a basic definition here: 
\begin{defi}\label{def:Heisen_pairs}
    Let $M \subseteq \RR^n$ be a manifold and $\Sigma \subseteq \RR^n$ be a set. We say that $(M, \Sigma)$  is a Heisenberg uniqueness pair if the only finite measure $\mu$ supported on $M$ with Fourier transform vanishing on $\Sigma$ is $\mu = 0$. 
\end{defi}
As a typical example of the kind of result from HUP that motivates us, we quote the following result:
\begin{theorem}[Theorem 2.1, \cite{F-BGJ}]\label{thm:FBGJ}
    Let $n \geq 2$ and $\Omega$ be a domain in $\RR^n$ with $0 \in \Omega$. Let $\theta_1, \theta_2 \in S^{n-1}$ be such that $\arccos\langle \theta_1, \theta_2\rangle \notin \pi \mathbb{Q}$. Let $k \in \RR$ and let $u$ be a solution of the Laplace–Helmholtz equation on $\Omega$:
    $$
    \Delta u+k^2u = 0.
    $$
    If $u$ satisfies one of the following boundary conditions 
    \begin{align*}
        u = 0 \text{  on } \left( \theta_1^\perp \cap \Omega \right) \cup  \left( \theta_2^\perp \cap \Omega \right)
    \end{align*}
    or
    \begin{align*}
        u =  0 \text{  on }  \theta_1^\perp \cap \Omega,~~  \text{ and }~~ \pa_\eta u = 0 \text{  on }  \left( \theta_2^\perp \cap \Omega \right),
    \end{align*}
    then $u \equiv 0$.
\end{theorem}
As is clear (and also mentioned by the authors in \cite{F-BGJ}), in dimension $n = 2$ the above theorem follows from  results in \cite{Ch}. We also observe that Theorem \ref{thm:ICC} above also shows that the angle between $\theta_j, j = 1, 2$ cannot be arbitrarily small (depending on the eigenvalue).

Now we begin proving Theorem \ref{thm:angle_estimate_high_dim}. Our proof is a modification of ideas in \cite{F-BGJ}.
\begin{proof}

Recall, we  would like to show that if the order of vanishing of $\varphi_\lambda$ at $p$ is $n_0$, then the angle between $M_1$ and $M_2$ at $p$,  $\arccos{\langle\eta_1, \eta_2 \rangle}\in P$, where $p$ lies in the intersection of two nodal hypersurfaces $M_1$ and $M_2$, and 
$$P=\left\{\frac{p}{q}\pi: q=1, 2, \cdots n_0, p= 0, 1, \cdots, q\right\}.$$
If possible, let $\arccos{\langle\eta_1, \eta_2 \rangle}\notin P$ where $\eta_i$, is a unit normal to $M_1$ at $p$. Without loss of generality, we can assume that $ p = 0$, the origin in $\RR^n$. Consider the spherical coordinates in $\RR^n$, $(r, \theta, \varphi)$ where $r\geq 0$, $\theta:= (\theta_1, \cdots, \theta_{n-2})\in [0, \pi)^{n-2}$, $\varphi\in [0, 2\pi)$.

It is known that, $\{Y_{\alpha}: \alpha\in \mathcal{S}:=  \NN_0^{n-2}\times \ZZ \}$ forms a basis of spherical harmonics where
$$Y_{\alpha}(r, \theta, \varphi)= r^{|\alpha|} \exp{(i\alpha_{n-1}\varphi)}\tilde{Y}_{\alpha}(\theta),$$
with $\displaystyle \tilde{Y}_{\alpha}(\theta) := \prod_{i=1}^{n-2}(\sin \theta_{n-i})^{|\alpha|^{i+1}}C_{\alpha_i}^{\gamma_i}(\cos \theta_i)$, and $|\alpha|^i= \alpha_i+ \alpha_{i+1 + \cdots+ \alpha_{n-1}}$, $\gamma_i= |\alpha|^{i+1}+(n-i-1)/2$, where  $C_{\alpha_i}^{\gamma_i}$ are the Gegenbauer polynomials. From the orthogonality of $C_n^\gamma$, for each $n \geq 1$, notice that the set 
$$
\left\{\tilde{Y}_{(\beta, m)}: (\beta, m)\in \NN_0^{n-1}, |\beta|+m=n \right\}, \NN_0 := \NN \cup \{ 0\}$$
is linearly independent.


Consider $V_\epsilon$ be an open ball around $0$. Then $M_i\cap V_\epsilon$ can be parametrized in polar coordinates as
$$
M_i\cap V_\epsilon= \left\{(r, \theta, \psi_i(r, \theta)), 0\leq r<\epsilon, \theta \in [0, \pi)^{n-2}\right\},
$$
where $\psi_i(r, \theta)\in S^1$ and $\psi_i$'s are smooth functions.

Defining $\varphi_i(\theta) : = \lim_{r\to \theta} \psi_i(r, \theta)$, from our assumption $\arccos{\langle\eta_1, \eta_2 \rangle}\notin P$, it follows that $\varphi_1-\varphi_2\notin P$. 


Since the order of vanishing of $\varphi_\lambda$ at $0$ is $n_0$, using Theorem \ref{Zildo}, the solution $\varphi_\lambda$ of $(\Delta+\lambda)\varphi_\lambda = 0$ in $M_i\cap V_\epsilon$ can be expressed in spherical coordinates in the form

\begin{equation}
    \varphi_\lambda(r, \theta, \psi_i(r, \theta))= r^{n_0}\sum_{m=-n_0}^{n_0}\left(\sum_{|\beta|+|m|=n_0} c_{\beta, m}\tilde{Y}_{\beta, m}(\theta) \right)e^{im\varphi_i} + o(r^{n_0}).
\end{equation}
Since, $\varphi_\lambda = 0$ on $M_i\cap V_\epsilon$, as $r\to 0$, it follows that,
\begin{equation*}
    \sum_{m=-n_0}^{n_0}\left(\sum_{|\beta|+|m|=n_0} c_{\beta, m}\tilde{Y}_{\beta, m}(\theta) \right)e^{im\varphi_i} = 0,
\end{equation*}
that is
\begin{equation}
    \sum_{|\beta|=n_0} c_{\beta, 0}\tilde{Y}_{\beta, 0}(\theta)+ \sum_{m=1}^{n_0}\left(\sum_{|\beta|+m=n_0} (c_{\beta, m}e^{im\varphi_i}+ c_{\beta, -m}e^{-im\varphi_i})\tilde{Y}_{\beta, m}(\theta) \right)=0
\end{equation}
Since $\{\tilde{Y}_{(\beta,m)}\}$ is linearly independent, $c_{\beta, 0}=0$ whenever $|\beta|=n_0$ and for each $m=1, 2, \cdots, n_0$, we have $n_0$ system of equations
$$c_{\beta, m}e^{im\varphi_1} + c_{\beta, -m}e^{-im\varphi_1}=0, $$
$$c_{\beta, m}e^{im\varphi_2} + c_{\beta, -m}e^{-im\varphi_2}=0. $$
Notice that, the determinant of each of the above systems is $2i\sin m(\varphi_1-\varphi_2), m=1, 2, \cdots, n_0$.
 If 
 $$(\varphi_1-\varphi_2)\notin \left\{\frac{p}{q}\pi: q=1, 2, \cdots n_0, p= 0, 1, \cdots, q\right\},$$
then each of above $n_0$ determinants is non-zero, which forces each $c_{\beta, m}=c_{\beta, -m}=0$, which implies that the coefficient of $r^{n_0}$ is zero. But this contradicts the fact that $\varphi_\lambda$ has $n_0$ order of vanishing at $0$. 
So,
$$\arccos{\langle\eta_1, \eta_2 \rangle}\in P.$$
\end{proof}
\begin{remark}
Recall the celebrated result of \cite{DF} that any $\lambda$-eigenfunction $\varphi_\lambda$ vanishes to at most order $c(M, g)\sqrt{\lambda}$ for any point in $M$. Also recall, from \cite{HS} that, the nodal set contains smooth $(n-1)$ dimensional submanifolds having finite $(n-1)$-dimensional measure in each compact subset of $\Omega$ and a closed countable $(n-2)$-rectifiable set. Then, using our result above, we have that whenever two 
nodal hypersurfaces intersect, the admissible angles between such intersecting hypersurfaces is 
from the set
\beq \label{def:set_P_adm_ang}
P =\left\{\frac{p}{q}\pi: q=1, 2, \cdots ,  [c\sqrt{\lambda}],\; p= 0, 1, \cdots, q\right\}.
\eeq
Also, using Theorem \ref{thm:ICC} or its interior case as in \cite{GM1}, we can rule out the cases when $p=0, q$ for every $q=1, 2, \cdots c\sqrt{\lambda}$. Then 
one sees that the minimum 
angle (in the sense of Theorem \ref{thm:angle_estimate_high_dim}) between two nodal hypersurfaces is $\gtrsim \frac{1}{\sqrt{\lambda}}$. 
\end{remark}

To sum up the discussion so far: consider a point $x \in \NNN_\varphi$. Then the opening angle of a nodal domain at $x$ will in general be given by Definition \ref{def:int_cone_cond}.  However, if $x$ happens to lie at the intersection of some nodal hypersurfaces, then the angle between any pair of such nodal hypersurfaces will  come from the set $P$ in (\ref{def:set_P_adm_ang}).

 \begin{figure}[ht]
\centering
\includegraphics[scale=0.18]{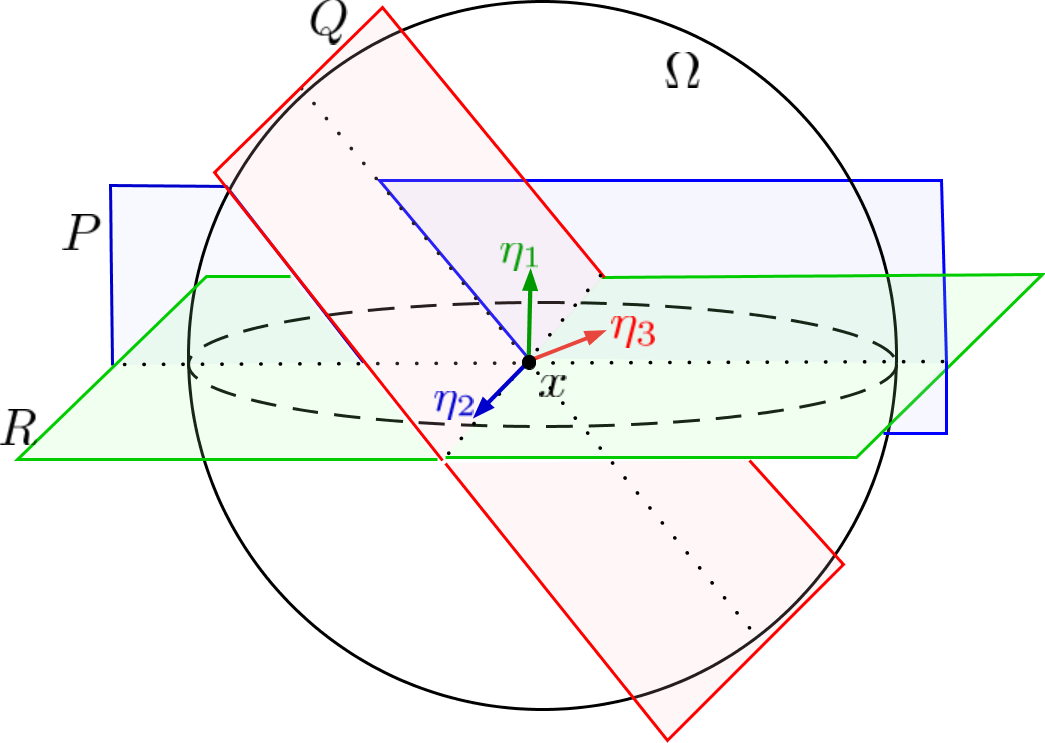}
\caption{$P, Q,$ and $R$ denote three nodal hypersurfaces of $\Omega$}
\label{fig:interpretation of opening angles}
\end{figure}

\section{Further applications: perturbation theory and low spectral gaps}\label{sec:perturbation} 

We study variation of Dirichlet Laplace spectrum and corresponding Laplace eigenfunctions with $C^2$ variations of bounded Euclidean domains, particularly those domains $\Omega$ which are critical points of $\lambda_2 - \lambda_1$.

Let the smooth deformation space of $\Omega$ be given by a Banach manifold $\Cal{B}$. First, we prove the following:
\begin{theorem}\label{thm:generic_spec}
The set of points inside $\Cal{B}$ (each represented by a perturbation of our starting domain $\Omega$) such that the Dirichlet Laplacian has simple spectrum is a residual set. 
\end{theorem}
We note that Theorem \ref{thm:generic_spec} is not new, for example see Example 3, Section 4 of \cite{U}. But we give our 
own proof, based on the perturbation formalism of \cite{GS}. 
%

Results of the nature of Theorem \ref{thm:generic_spec} 
are ultimately based on  transversality phenomena (as illustrated in 
\cite{U}). Loosely speaking, they can be considered infinite dimensional analogues of the following statement: generically, all symmetric matrices have non-repeated 
eigenvalues. At a more basic level, an equivalent statement is the fact that single variable polynomials generically have non-repeated roots. 

\subsection{Proof of Theorem \ref{thm:generic_spec}}\label{subsec:prelim}
The topic of variation of spectra under perturbation has a long history starting with the analytic perturbation theory of 
Kato (see \cite{Ka}). In the case of rather generic families of elliptic operators, see pioneering work in \cite{A} and \cite{U}. 
In case the perturbation is non-generic, such results have been recently studied in, for example, \cite{HJ1, HJ2, Mu} etc. In this note, we give a slight variant of a proof for Theorem \ref{thm:generic_spec}. To set up the stage, we start by considering a bounded domain $\Omega \subset \RR^n$, and consider a vector field $V$ defined on $\RR^n$, whose coordinates we denote by 
$(V_1,..., V_n)$, and whose regularity we assume to be $C^2$ for immediate purposes. Now, consider the perturbation 
of the domain $\Omega$ along the vector field $V$ to the domain $\Omega_\varepsilon$, defined by 
$\{ x^\varepsilon = x + \varepsilon V : x \in \Omega\}$. We wish to study the variation of the eigenequation 
\beq\label{eq:eigen}
-\Delta \varphi = \lambda \varphi
\eeq 
along the parameter $\varepsilon$. However, to fit the language of perturbation theory, instead of dealing with a one-parameter family of domains, it is much more convenient to pull back all $\Omega_\varepsilon$ 
to the original domain $\Omega$, so that we get a one-parameter family of elliptic PDEs on $\Omega$ whose coefficients are dependent on 
$\varepsilon$. This lands us in the familiar framework of a family of self-adjoint operators with common domain of definition 
varying over a Banach manifold. Upon computation following \cite{GS} (see Sections 4, 5 and particularly pp. 299, Section 6), 
we see that the eigenequation 
$$
-\Delta_\varepsilon \varphi_\varepsilon = \lambda_\varepsilon \varphi_\varepsilon
$$ 
on $\Omega_\varepsilon$, when pulled back to $\Omega$, becomes

\beq\label{eq:pull_back_eigen}
A_\varepsilon u = \sum_{j, k}-\pa_k (J\beta_{kj}\pa_j u) = \lambda_\varepsilon Ju,
\eeq
with the Dirichlet boundary condition being preserved, and where $J$ is the determinant of the Jacobian matrix of the transformation 
$x \mapsto x^\varepsilon$, and $\beta_{jk} = \sum_l \frac{\pa x_j}{\pa x^\varepsilon_l}\frac{\pa x_k}{\pa x^\varepsilon_l}$. 
To see this, write 
\begin{align*}
(A_\varepsilon u, v)_{L^2(\Omega)} & = (-\Delta_\varepsilon u^\varepsilon, v^\varepsilon)_{L^2(\Omega_\varepsilon)} = \int_{\Omega_\varepsilon} -\Delta_\varepsilon u^\varepsilon v^\varepsilon dx^\varepsilon\\
& = \sum_k\int_{\Omega_\varepsilon} \pa_{x_k^\varepsilon} u^\varepsilon \pa_{x_k^\varepsilon} v^\varepsilon dx^\varepsilon = \sum_{i, j, k}\int_\Omega \pa_{x_j} u \frac{\pa x_j}{\pa_{x_k^\varepsilon}} \pa_{x_i} v \frac{\pa x_i}{\pa_{x_k^\varepsilon}} J dx\\
& = \int_\Omega -\pa_{x_i}(J \frac{\pa x_i}{\pa_{x_k^\varepsilon}}\frac{\pa x_j}{\pa_{x_k^\varepsilon}} \pa_{x_j}u)\; v\; dx.
\end{align*}
The main idea behind the computation is that since
$$
\frac{\pa x^\varepsilon_j}{\pa x_i} = \delta_{ij} + \varepsilon \frac{\pa V_j}{\pa x_i} 
$$
up to first order errors in $\varepsilon$, we can write that 
$$
\frac{\pa x_j}{\pa x^\varepsilon_k} = \delta_{jk} - \varepsilon \frac{\pa V_j}{\pa x_k} + O(\varepsilon),
$$
whereas $J$ can be expressed as
$$
J = 1 + \varepsilon (\sum_j \frac{\pa V_j}{\pa x_j}) + O(\varepsilon),
$$
and
$$
\beta_{jk} = \delta_{jk} - \varepsilon\left( \frac{\pa V_j}{\pa x_k} + \frac{\pa V_k}{\pa x_j}\right) + O(\varepsilon),
$$
and $\beta_{jk}$ has a power series expansion 
$$
\beta_{jk} = \delta_{jk} + \sum_{i = 1}^\infty \varepsilon^i \beta^i_{jk},
$$
where $\beta^1_{jk} = - \left( \frac{\pa V_j}{\pa x_k} + \frac{\pa V_k}{\pa x_j}\right)$, as mentioned before. For details on the above, see \cite{GS}, Section 6.

All told then, the perturbation $A_\varepsilon$ can be expressed as 
\beq\label{eq:var_eps_op}
A_\varepsilon u = -\Delta u + \varepsilon \left( \sum_{j, k} \pa_k \left( (\pa_k V_j + \pa_j V_k) \pa_j u\right) + \frac{1}{2} \sum_j \pa_j u 
\Delta \left(  V_j\right)\right) + O(\varepsilon).
\eeq

Now we bring in the Sard-Smale transversality formalism used by Uhlenbeck. We first quote the theorem:
\begin{theorem}\label{thm:Sard_Smale}
Let $\Phi : H \times B \to E$ be a $C^k$ map, where $H, B$ and $E$ are Banach manifolds with $H$ and $E$ separable. If $0$ is a regular 
value of $\Phi$ and $\Phi_b := \Phi (., b)$ is a Fredholm map of index $< k$, then the set $\{ b \in B : 0 \text{  is a regular value of  }
\Phi_b\}$ is residual in $B$.
\end{theorem}

Here, we wish to check Theorem \ref{thm:Sard_Smale} for our domain perturbations in the particular setting that $H = E = \Cal{D}(\Delta)$ and $B$ is the collection of parameters for domain perturbation. For starters, all $A_\varepsilon$ are self-adjoint by the Kato-Rellich theorem, being relatively bounded perturbations of $A_0 = \Delta_\Omega$. Also, ellipticity in such cases implies the Fredholm property, as is well-known.

Now, if $0$ is not a regular value as above, we have that
for all perturbations given by $V$ and $\varepsilon$ small, we have Laplace eigenfunctions $\varphi, \psi$ ($\psi$ corresponding to the eigenvalue 
$\lambda$) such that
\begin{align*}\label{eq:lin_dep_trans}
-\int_\Omega \sum_{jk} \left( \left( \pa_k V_j + \pa_j V_k\right) \pa_j \varphi \pa_k \psi\right) + \frac{\varphi\psi}{2} \Delta \left(
\sum_j \pa_j V_j\right) & = - \int_{\pa \Omega} \frac{\pa \varphi}{\pa \eta}\frac{\pa \psi}{\pa \eta} \left(\sum_k V_k\eta_k\right)\\
& = 0. 
\end{align*}
This basically means that by Holmgren's uniqueness theorem, $\varphi$ and $\psi$ are identically zero, establishing our claim.

\subsection{Discussion on some variants of Theorem \ref{thm:generic_spec}}\label{subsec:gen_spec_convex}
Above we discussed the generic spectral simplicity of Euclidean domains in $\RR^n$. However, it is a valid question to ask what kind of generic properties the spectrum has provided the word ``generic'' is constrained on a much smaller moduli space. For example, we refer the readers to recent work in \cite{HJ1, HJ2}. Here, we outline a short result to illustrate the line of thought in these later works. Consider the family $\Cal{F}$ of all domains $\Omega$ in the plane whose boundary can be written as $\pa \Omega = \Cal{S} \cup \Cal{C}$, where 
\begin{enumerate}
    \item $\Cal{S}$ is a disjoint collection of straight line segments and $\Cal{C}$ is a disjoint collection of strictly curved real analytic line segments, and
    \item $\Cal{S} \cap \Cal{C}$ is a finite collection of points.
\end{enumerate}
\begin{claim}
The subcollection of domains in $\Cal{F}$ which have simple Dirichlet spectrum is a residual set.
\end{claim}
\begin{proof}

We would like to show that given any $k$, the subfamily of $\Cal{F}$ whose first $k$ Dirichlet eigenvalues are non-repeated form a residual set. Since the intersection of countably many residual sets is residual, the sub-family of $\Cal{F}$ that has simple Dirichlet spectrum is also residual.  

Since any two members $\Omega_1, \Omega_2$ of $\Cal{F}$ which satisfy $|\Cal{S}_1| = |\Cal{S}_2|$ and $|\Cal{C}_1| = |\Cal{C}_2|$ can be joined by a one-parameter family of real-analytic maps, and the Rayleigh quotient varies analytically under analytic perturbations, one sees immediately that given $i < j \leq m$, it is enough to find one $\Omega$ in this family for which $\lambda_i \neq \lambda_j$. 

Now, take any such member of this family, and find a rectangle $R$ with simple Dirichlet spectrum (it is easy to check that for a rectangle with side lengths $l_1, l_2$ this happens if and only if $\left(\frac{l_1}{l_2}\right)^2 \notin \mathbb{Q}$). Now, start by inscribing $\Omega$ inside $R$ (by scaling if necessary), and consider a one-parameter family of analytic perturbations $\Omega_t$ which converge to $R$. By spectral convergence, it is clear that for any given $m$, there exists a large enough $t$ such that the first $m$ Dirichlet eigenvalues of $\Omega_t$ are non-repeated.

\end{proof}

We end this section with the following question and we believe that an affirmative answer would also help understand the Payne property better for future studies.
\begin{question}
    Given a domain $\Omega \subseteq \RR^n$ or a closed manifold $M$, and let $\mathfrak{M}(\Omega)$ and $\mathfrak{M}(M)$  denote the moduli space of all smooth perturbations of $\Omega$ and all Riemannian metrics on $M$ respectively. Are these moduli spaces path connected?
\end{question}

\subsection{Fundamental gap, narrow convex domains and small perturbations}\label{subsec:gap_proof}

Now we take a look at the problem of minimising the fundamental gap. In this regard, recall the main result from \cite{AC}: for any convex domain $\Omega\subset\RR^n$,
$$
\lambda_2 - \lambda_1 \geq \frac{3\pi^2}{D^2}, 
$$
where $D = \diam \Omega$. Now, the following question is natural: 

\begin{question}
    Is the above inequality saturated by some domain? 
\end{question}
The popular belief in the community seems that it is not, and any infimising sequence for $\lambda_2 - \lambda_1$ (under the normalisation $D = 1$) should degenerate to a line segment. In particular, the correct regime to look for in the search for minimisers is the class of narrow convex domains. This problem seems quite difficult, as standard precompactness ideas (e.g., see recent work in \cite{MTV, KL})  do not apply directly. Also, it is quite resistant to perturbative techniques, as generic perturbations (even small ones) might destroy convexity. In addition, the problem seems quite sensitive to the class of domains: it might demonstrate a markedly different behaviour if the overall class of domains is changed, for example see \cite{LR}. 


Recall that $\mathfrak{P}$ denotes the class of strictly convex $C^2$-planar domains. Now we begin proving Theorem \ref{thm:fund_gap_perturb}. We finish the proof in two steps. First, we prove the following:

\begin{theorem}\label{thm:gap_1}
	Let $\Omega \in \mathfrak{P}$ 
	with  diameter $D = 1$ and inner radius $\rho$ which minimises the fundamental gap functional $\lambda_2 - \lambda_1$ in 
	$\mathfrak{P}$. There exists a universal constant $C \ll 1$ such that if $\rho \leq C$, 
	then $\lambda_2(\Omega)$ is not simple.
\end{theorem}

\begin{proof}
Recall the Hadamard formula (see Section 2.5.2 of \cite{He}) which expresses the evolution of Laplace spectrum with respect to perturbation of a domain $\Omega$ by a vector field $V$:
\beq\label{eq:eigen_evol}
\lambda'_k(0) = - \int_{\pa \Omega} \left( \frac{\pa \varphi}{\pa \eta}\right)^2 V\cdot\eta \; dS.
\eeq
Suppose $\lambda_i, \lambda_j$ are Dirichlet eigenvalues of $\Omega$ with corresponding eigenfunctions $\varphi_i, \varphi_j$ respectively. If $\lambda_j - \lambda_i$ considered as a function of domains has a critical point at a domain $\Omega$, then we must have that 
\beq
0 = (\lambda_j - \lambda_i)'(0) = -\int_{\pa \Omega} \left(\left( \frac{\pa \varphi_j}{\pa \eta}\right)^2 -  \left( \frac{\pa \varphi_i}{\pa \eta}\right)^2 \right) V\cdot\eta \; dS,
\eeq
for all perturbation vector fields $V$. Now we specify to the special case $j = 2, i = 1$, and the above calculation with the Hadamard formula holds  true under the assumption that $\lambda_2$ is simple. 

Now, consider a small perturbation vector field $V$ such that $V = 0$ at $x_1, x_2$ (see diagram below) and $V\cdot\eta$ is non sign-changing away from $x_1,x_2$. Note that the aforementioned $V$ constrains the diameter to be fixed along the perturbation. Then, this implies that 
$$
\left| \frac{\pa \varphi_1}{\pa \eta}\right| = \left| \frac{\pa \varphi_2}{\pa \eta}\right| \text{   on   } \pa \Omega \text{ away from } \{x_1, x_2\}.
$$
Using the maximum principles 
and the Hopf Lemma, we find that without loss
of generality, $\frac{\pa \varphi_1}{\pa \eta}>0$ on $\pa \Omega$. 
Moreover, since 
convex domains or their small perturbations satisfy the strong Payne property (as proved in \cite{M, A, MS}), using Lemma 1.2 of \cite{Lin} there exist exactly two points $p_1, p_2\in \pa \Omega$ such that  
$$\varphi_2(p_1)=\varphi_2(p_2)=0, \quad \text{and } \quad\frac{\pa \varphi_2}{\pa \eta}(p_1)= \frac{\pa \varphi_2}{\pa \eta}(p_2)=0.$$

 \begin{figure}[ht]
\centering
\includegraphics[scale=0.17]{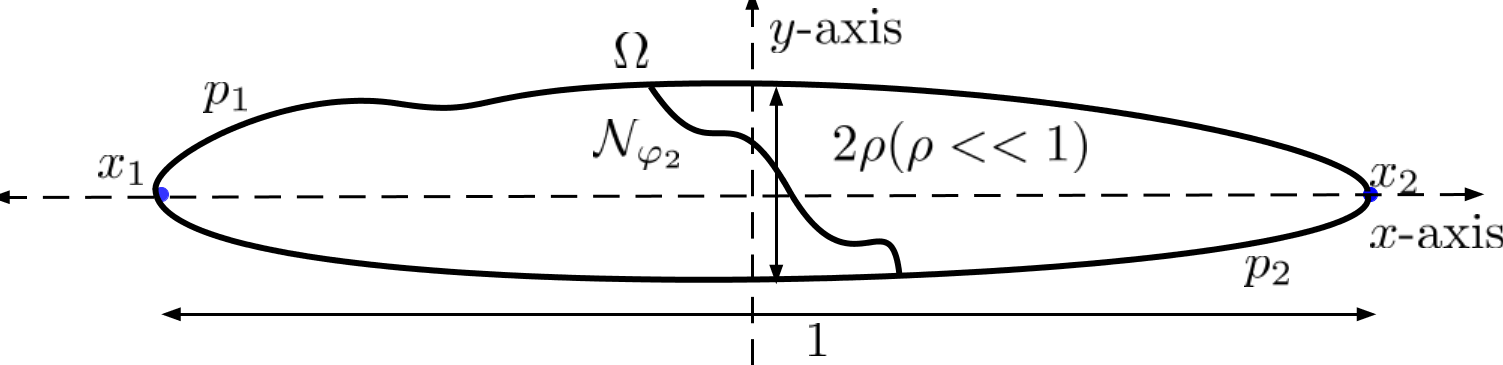} 
\caption{Small perturbation of a narrow convex domain}
\end{figure}
     

From work in \cite{J, J1, J2} (and perturbation arguments based on \cite{MS}), it is known that 
$p_1, p_2$ cannot be near $x_1, x_2$ once the domain $\Omega$ is long and narrow enough (which is encoded in the statement by the universal constant $C$). This is a contradiction since
$$0=\left|\frac{\pa \varphi_2}{\pa \eta}(p_i)\right|=\left|\frac{\pa \varphi_1}{\pa \eta}(p_i)\right|>0.$$

\end{proof}

\begin{que}
    The following interesting question comes up in connection to the proof of the last theorem. On a domain $\Omega$, can there be two Dirichlet eigenfunctions (corresponding to different eigenvalues) such that they also have the same Neumann data? One is tempted to speculate that such an event should not happen unless $\Omega$ is a ball. As pointed out by Antoine Henrot, the first and sixth eigenfunctions on the planar disc are both radially symmetric, so they can be scaled to have the same Neumann data.  This is in turn related to a conjecture due to Schiffer. 
\end{que} 


We augment the above result by the following observation. 
\begin{theorem}\label{thm:matrix_sig}
Let $\Omega \subset \RR^n$ be a $C^2$-domain. Assume that $\Omega$ has a multiple eigenvalue of the Dirichlet Laplacian 
$$
\lambda_{k + 1}(\Omega) = \lambda_{k + 2}(\Omega) = \cdots = \lambda_{k + m}(\Omega).
$$
Then for each fixed $1 \leq l \leq m$ there exists a deformation field $\Omega_t$ passing through $\Omega_0 := \Omega$ generated by a $C^2$-vector field $V$ such that for small enough $t$, 
$$
\lambda_{k + 1}(\Omega_t) < \lambda_{k + 1}(\Omega_0),\cdots,  \lambda_{k + l}(\Omega_t) < \lambda_{k + l}(\Omega_0),$$
and 
$$
\lambda_{k + l + 1}(\Omega_t) > \lambda_{k + l + 1}(\Omega_0),\cdots,  \lambda_{k + m}(\Omega_t) > \lambda_{k + m}(\Omega_0).$$
Furthermore, we can ensure that
$$ 
|\Omega_0| = |\Omega_t|.
$$
Suppose $\Omega_0\in \mathfrak{P}$ is a long narrow domain as in Theorem \ref{thm:gap_1} above. Then, we can additionally ensure that $\Omega_t\in \mathfrak{P}$ and
$$
\diam(\Omega_0) = \diam(\Omega_t).
$$
\end{theorem}
The proof is based on some ideas in Lemma 1 of  \cite{HO}. \newline
\begin{proof} 
We begin by observing that because of the existence of multiple eigenvalues, 
$\lambda_{k + p}(\Omega_t)$ is not differentiable at $t = 0$ in the usual Frechet sense, but there is a nice formula giving directional derivatives, in the sense of the limit $\displaystyle{\frac{\lambda_{k + p}(\Omega_t) - \lambda_{k + p}(\Omega_t)}{t} }$ as $t \to 0$. Such directional derivatives are precisely the eigenvalues of the $m \times m$-matrix
\beq\label{eq:direction_der}
\Cal{M} = \left( -\int_{\pa \Omega} \frac{\pa u_i}{\pa \eta} \frac{\pa u_j}{\pa \eta} V\cdot\eta \; d\sigma \right), \quad p = 1, 2,\cdots, m,
\eeq
where $u_i, 1 \leq i \leq m$ denotes the eigenspace for the repeated eigenvalue $\lambda_{k + 1}$. \newline
Let us consider points $A_1, A_2,\cdots, A_m \in \pa\Omega$ the choice of which will be explained below. Also, consider a deformation vector field $V$ such that 
$V\cdot\eta = 1$ in a $\varepsilon$-neighbourhood of $A_1, A_2,\cdots, A_l$ and 
$V\cdot\eta = -1$ in a $\varepsilon$-neighbourhood of $A_{l + 1},\cdots, A_m$, $V.\eta=0$ outside a $2\epsilon$-neighbourhood of $A_1, \cdots, A_m$, and regularized in a $2\varepsilon$-neighbourhood around each such point maintaining $|\Omega_0| = |\Omega_t|$.  To preserve the diameter also, one just needs to choose the points $A_j$ sufficiently away from $x_1, x_2$ (see the figure above).

By (\ref{eq:direction_der}) above, it suffices to prove that the symmetric matrix $\Cal{M}$ has signature $(l, m - l)$. When $\varepsilon \to 0$, $\Cal{M}$ converges to the matrix
$$
M = \left( - \sum_{k = 1}^l \frac{\pa u_i}{\pa \eta}(A_k)\frac{\pa u_j}{\pa \eta}(A_k) + \sum_{k = l + 1}^m  \frac{\pa u_i}{\pa \eta}(A_k)\frac{\pa u_j}{\pa \eta}(A_k)\right).
$$
Consider the column vectors $v_{A_k} := (\frac{\pa u_1}{\pa \eta}(A_k), \cdots, \frac{\pa u_m}{\pa \eta}(A_k))^T$. Note that $M=V\cdot W$, where 
$$V= (v_{A_1}, \cdots, v_{A_m})\quad \text{ and,} \quad W=(-v_{A_1, }, \cdots, -v_{A_l}, v_{A_{l+1}}, \cdots, v_{A_m} )^T.$$
It is enough to ensure that the vectors $\{v_{A_k}:k=1,\cdots,m \}$ are linearly independent. Then the signature of $M$ is $(l, m-l)$. 

If the columns in the matrix $V$ are not independent, then they satisfy a homogeneous linear equation, which means in turn that there is a homogeneous linear relation among the rows of $V$, namely that, 
on an open set $S \subset \pa \Omega$ away from $x_1, x_2$, we have that 
$$
\sum_{p = 1}^m c_p  \frac{\pa u_p}{\pa \eta} =   \frac{\pa}{\pa \eta} \left(\sum_{p = 1}^m c_p u_p\right) = 0 \quad \text{on } S.
$$
This is a contradiction from H\"{o}lmgren's uniqueness theorem.

\end{proof}
The following question seems interesting:
\begin{question}
    Could we also ensure that $\lambda_{k + l}(\Omega_t) = \lambda_{k + l}(\Omega_0) < \lambda_{k + l +1}(\Omega_t)$ at the expense of changing the volume of $\Omega_t$?
\end{question}

\begin{proof}[Proof of Theorem \ref{thm:fund_gap_perturb}]
Now, if we observe the way that the vector field was chosen in the proof of Theorem \ref{thm:matrix_sig}, it is clear that if one wants to fix only the diameter, one can choose such a $V$ easily such that $\lambda_1(\Omega_t) > \lambda_1(\Omega_0)$ for small enough $t$. This reduces the gap even further, contradicting that $\Omega_0$ is a minimiser. Finally, putting Theorems \ref{thm:gap_1} and \ref{thm:matrix_sig} together, we conclude the proof. 
\end{proof}

\begin{remark}
Theorem \ref{thm:matrix_sig} is not essential for the proof of Theorem \ref{thm:fund_gap_perturb}, but it might be of independent interest to few. The topological restrictions on the first nodal set imposed by Jerison in \cite{J} is enough to prove Theorem \ref{thm:fund_gap_perturb}, and the idea of the proof is as follows: from Theorem \ref{thm:gap_1}, we know that the second eigenvalue of any minimising domain $\Omega$ cannot be simple. Let $\varphi_1, \varphi_2$ be any two linearly independent second eigenfunctions of $\Omega$. From choosing any point $p\in \Omega$ sufficiently close to the ``ends'' $x_1$ or $x_2$, there exists a second eigenfunction $\varphi=c_1\varphi_1+c_2\varphi_2$ (for some $c_1, c_2\in \RR$) such that $\varphi(p)=0$. This leads to  contradiction, since from \cite{J} we know that the first nodal set stays away from the ends $x_1, x_2$. This in turn implies that $\Omega$ cannot be a minimiser, which concludes the proof of Theorem \ref{thm:fund_gap_perturb}. 
\end{remark}

\section{Appendix: eigenvalue multiplicity, nodal intersection/detachment and topology}
In this appendix we make some general remarks about the multiplicity of eigenvalues. An important component in the proof of \cite{M} was a result from \cite{Lin} which states that if the Neumann data of a second Dirichlet eigenfunction is non sign-changing on the boundary of a convex planar domain $\Omega$, then the second eigenvalue of $\Omega$ is simple. So, studying the multiplicity of eigenvalues might prove important for future work on the Payne conjecture. In this regard, we begin with the following:
\begin{proposition}
If $\Omega \subseteq \RR^n$ is a simply connected 
domain, 
and $\varphi$ and $\psi$ are two Dirichlet eigenfunctions corresponding to the same eigenvalue, then every connected component of $\NNN_\psi$ intersects $\NNN_\varphi$ at at least one point.
\end{proposition}
\begin{proof}
Let $\Omega_\varphi$ be a nodal domain for $\varphi$. Then, 
\begin{align*}
    \int_{\Omega_\varphi} \Delta \varphi\; \psi - \varphi\;\Delta \psi & = \int_{\pa\Omega_\varphi} \frac{\pa\varphi}{\pa \eta}\psi - \frac{\pa\psi}{\pa \eta}\varphi \\
    & = \int_{\pa\Omega_\varphi} \frac{\pa\varphi}{\pa \eta}\psi = 0,
\end{align*}
which implies the sign change of $\psi$ on $\pa \Omega_\varphi$, and hence the result. Observe that we have used the fact that on a simply connected domain, $\pa \Omega_\varphi$ has to be connected.
\end{proof}

Observe that the above is not true when the domain has more complicated topology. We refer the reader to \cite{Do} where, following earlier work in \cite{Gi}, it is proved that on a closed Riemannian manifold $M$, if there exist two eigenfunctions for the same eigenvalue which do not simultaneously vanish, then $H_1(M) \neq 0$. An illustrative example could be that of a torus (which is a surface of revolution in $\RR^3$), where the nodal set can be ``pushed'' or translated using the rotational isometry. However, it is an interesting question whether such a result could be true on Euclidean domains. The argument in Donnelly is too pretty not to mention, and we add to it our  observation that his argument would also extend to Euclidean domains with Neumann boundary conditions. 

\begin{proposition}\label{prop:Neumann_Don_Gichev}
    Let $\Omega \subset \RR^n$ be a bounded domain with Neumann boundary conditions. Suppose as before that $\varphi$ and $\psi$ are two eigenfunctions corresponding to the same eigenvalue $\mu$. If their nodal sets do not intersect, then $H_1(\Omega) \neq 0$.
\end{proposition} 

\begin{proof}
The proof is essentially a minor observation on the proof in \cite{Do}. Let $u := \varphi + i\psi$, and $X := \Im \frac{\nabla u}{u}$. By our assumption, the denominator is never zero. On calculation, it can be checked that the flow of the vector field $Y := \frac{X}{\sqrt{1 + |X|^2}}$ leaves invariant the finite measure $\frac{|\psi|^2}{\sqrt{1 + |X|^2}} \; dV$. On a closed manifold, this is a complete flow. On a manifold with boundary to use the same idea, we need the Neumann boundary condition, as then the gradient of the eigenfunctions is tangential to the boundary. Now, by Poincar\'{e} recurrence, there is an integral curve $\alpha$ of the vector field $Y$ which returns arbitrarily close to its starting point. It can be completed to a closed path $\alpha$ such that $\int_\alpha \Im \frac{du}{u} \gtrsim 1$, which means that both $ \varphi(\alpha)$ and $\psi (\alpha)$ are not zero homologous in $\CC^*$.  
\end{proof}
This raises the natural question for the corresponding case of Dirichlet boundary conditions. It is again a quick observation that if $M$ is a compact manifold with totally geodesic boundary $\pa M$, then the double $\tilde{M}$ of $M$ can be constructed with $C^2$ Riemannian metric (see \cite{Mo}) where the Dirichlet eigenfunctions can be continued by ``reflection about the boundary''. Clearly in this case, obvious modifications of the argument from \cite{Do} will still work. However, it is not clear how to proceed in the case of a Euclidean domain, where the corresponding statement still seems intuitively true. 

So, for the case of Dirichlet boundary conditions, instead of taking the approach of \cite{Do}, it seems fruitful to revert back to the original approach of Gichev. A careful scan of the ideas in  
\cite{Gi} reveals the following:
\begin{proposition}\label{prop:Dirichlet_Don_Gichev}
    Let $\Omega \subset \RR^n$ be a bounded domain and consider the Laplacian on $\Omega$ with the Dirichlet boundary condition. If $\varphi$ and $\psi$ are two eigenfunctions corresponding to the same eigenvalue $\lambda$ whose nodal sets do not intersect, then $H^1(\Omega) \neq 0$.
\end{proposition}
The proof is verbatim similar to the one in \cite{Gi}, with obvious modifications to allow for the boundary, and we skip it. 

\begin{remark}
    In light of the assumptions imposed on all the results of this section, it is natural to wonder when the Laplace-Beltrami operator of a closed manifold has repeated spectrum. Firstly, it is a well-known fact that {\em generic} spaces have simple spectrum. This is a well-known transversality phenomenon investigated in \cite{Al, U} (we also give our own proof in Section \ref{sec:perturbation} below). On the other hand, it is a well-known heuristic (by now folklore) that the presence of symmetries of the space $M$ leads to repetitions in the spectrum. An explicit proof of this heuristic using a variant of the Peter-Weyl argument has been recorded in \cite{Ta}. The main claim is that the presence of a non-commutative group $G$ of isometries of the space will lead to infinitely many repeated eigenvalues of the Laplace-Beltrami operator.
\end{remark}

\begin{remark}\label{rem:one_pt_intersection}
It is clear that on a simply connected domain, satisfaction of the strong Payne property implies that the multiplicity of the second Dirichlet eigenvalue is at most two. Suppose there are three eigenfunctions $\varphi_j, j = 1, 2, 3$ corresponding to $\lambda_2$. Then picking any two points $p, q \in \pa\Omega$, one can find an eigenfunction $\psi_{pq} = \sum_j \alpha_j \varphi_j$ such that $\psi$ intersects $\pa\Omega$ exactly at $p, q$. Now, consider a sequence $p_n, q_n$ approaching a common point $o \in \pa\Omega$. Then in the limit one gets an eigenfunction $\psi$ which intersects $\pa\Omega$ at exactly $o$, and vanishes to order at least $2$ there, which would be a contradiction of the strong Payne property.
\end{remark}

With that in place, we look at the following multiplicity result. 

\begin{theorem}\label{thm:mutliplicity_second_eigen}
    Consider a bounded simply connected domain $\Omega \subseteq \RR^2$ 
   and let $(0,0)\notin \Omega$  satisfying the following:
    \begin{itemize}
       \item the boundary $\pa \Omega$ contains exactly two distinct points $P, Q$ dividing $\pa \Omega$ into two components $\Gamma_j, j = 1, 2$ such that the outward unit normal at $P$ (respectively, $Q$) is in the direction of the vector joining $(0,0)$ to $P$ (respectively, vector joining $Q$ to $(0,0)$).
       
        \item at every point $(x, y) \in \Gamma_1$, the outward normal $\eta$ makes an acute angle with $(-y, x)$ and at every point $(x, y) \in \Gamma_2$, the outward normal $\eta$ makes an obtuse angle with $(-y, x)$.
    \end{itemize}
    In such domains, the multiplicity of the second Dirichlet eigenvalue is at most $2$.
\end{theorem}
\begin{remark}
Observe that the above theorem includes in particular domains which are convex in one direction, when the origin is taken arbitrarily far from the domain (point at infinity).
\end{remark}

\begin{proof}
Let $\Omega\subset\RR^2$ be a bounded simply connected domain. Let $(0,0)\notin \Omega$ be a point such that 
the boundary $\pa \Omega$ contains exactly two distinct points $P, Q$ dividing $\pa \Omega$ into two components $\Gamma_j, j = 1, 2$ in a way that the unit outward normal  at $P$ is in the direction to the vector joining $(0,0)$ to $P$  and the unit outward normal (blue arrows in Figure \ref{fig:domain with mult. atmost 2}) at $Q$ is opposite to the vector joining $(0,0)$ to $Q$  (green arrows in Figure \ref{fig:domain with mult. atmost 2}). In particular, for the points $P=(P_1,P_2)$ and $Q=(Q_1, Q_2)$ we have 
$$
\left\langle \eta_{P}, \frac{(P_1,P_2)}{|P|} \right\rangle= 1, \quad \text{and} \quad \left\langle \eta_{Q}, \frac{(Q_1,Q_2)}{|Q|} \right\rangle = -1.
$$
Considering the rotational vector field 
$$X=-y\frac{\pa}{\pa x}+ x\frac{\pa }{\pa y},$$
from our second assumption we have that,  $\langle\eta_{(x,y)}, X_{(x,y)} \rangle  > 0  \text{ (respectively, } < 0)$ when  $(x,y)\in \Gamma_1$ (respectively, $\Gamma_2$), and  $\langle\eta_{P}, X_{P} \rangle  = \langle\eta_{Q}, X_{Q} \rangle =0. $



Suppose to the contrary, that the multiplicity of $\lambda_2$ is at least $3$. Up to forming linear combinations, let $\psi$ be a second eigenfunction whose nodal set intersects $\pa\Omega$ at $P, Q$. Also, given $O\in \pa \Omega$, one can find a second eigenfunction $\varphi$  whose nodal set intersects $\pa \Omega$ at $O$, as in Remark \ref{rem:one_pt_intersection} (see diagram below). 
\vspace{2mm}
 \begin{figure}[ht]
\centering
\includegraphics[scale=0.18]{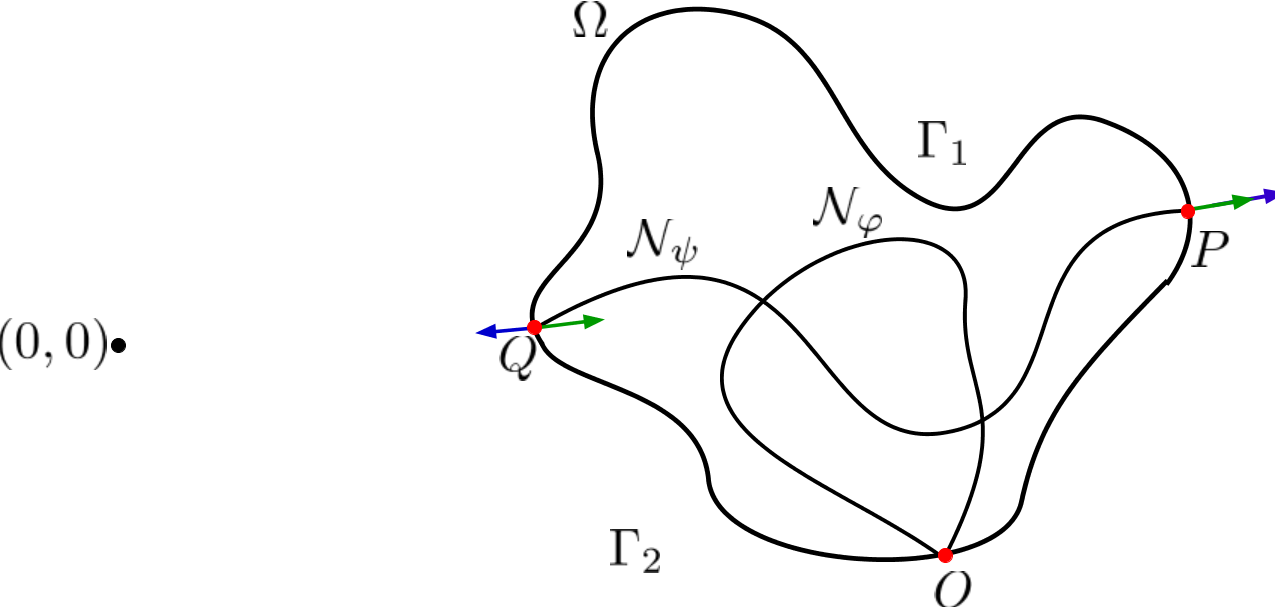}
\caption{A non-convex domain satisfying the assumptions of Theorem \ref{thm:mutliplicity_second_eigen}}
\label{fig:domain with mult. atmost 2}
\end{figure}
     
   \vspace{2mm}
  
One can check that $[\Delta, X]=0$, which gives us that 
\begin{align*}
   \int_\Omega \varphi\Delta(X\psi) - X\psi\Delta \varphi & = \int_\Omega \varphi X \Delta \psi - X\psi\Delta \varphi \\
   & = \int_\Omega -\lambda_2\varphi X\psi  - X\psi\Delta \varphi\\
   & = \int_\Omega \Delta \varphi X\psi - \Delta \varphi X\psi \\
   & = 0.
\end{align*}
Then we  have
\begin{align*} 
 0=  \int_\Omega \varphi\Delta(X\psi) - X\psi\Delta \varphi & = \int_{\pa\Omega} X\psi\;\frac{\pa\varphi}{\pa\eta} = \int_{\pa\Omega} \langle X, 
   \eta\rangle\; \frac{\pa\psi}{\pa \eta}\;\frac{\pa\varphi}{\pa\eta}. 
\end{align*}
Since $\langle X, \eta\rangle\frac{\pa\psi}{\pa \eta}$ does not change sign on $\pa\Omega$, $\frac{\pa\varphi}{\pa\eta}$ must which leads to a contradiction.
\end{proof}


\subsection{Acknowledgements} The authors are grateful to IIT Bombay for providing ideal working conditions. The research of the first author was partially supported by SEED Grant RD/0519-IRCCSH0-024. During the final stages of the preparation of the manuscript the first author was a visitor at MPIM Bonn. The second named author would like to thank the Council of Scientific and Industrial Research, India for funding which supported his research. The second named author would also like to thank Iowa State University for their support during the final stages of the preparation of the manuscript. The authors would like to acknowledge useful conversations and correspondence with Stefan Steinerberger and Antoine Henrot. The authors are also deeply grateful to Anup Biswas and Daniel Peralta-Salas for pointing out  important corrections in the first version of the paper.

\end{document}